\documentclass[12pt]{article}

\usepackage{hyperref}
\hypersetup{
  colorlinks   = true,    
  urlcolor     = blue,    
  linkcolor    = blue,    
  citecolor    = red      
}
\usepackage{epstopdf}
\usepackage[caption=false]{subfig}

\usepackage{longtable}
\usepackage{amsmath}
\usepackage{csquotes}
\usepackage{tikz}
\usepackage{upgreek}
\usepackage{amssymb}
\usepackage{amsthm}
\usepackage{amsfonts}
\numberwithin{equation}{section}
\usepackage{booktabs,caption}
\usepackage[flushleft]{threeparttable}
\usepackage{stmaryrd}
\usepackage{wasysym}
\usepackage{makecell}
\usepackage{float}
\usepackage{graphicx}
\usepackage{lipsum}
\usepackage{multirow}
\usepackage{bm}

\newtheorem{theorem}{Theorem}[section]
\newtheorem{definition}[theorem]{Definition}

\newtheorem{lemma}[theorem]{Lemma}

\newtheorem{corollary}[theorem]{Corollary}
\newtheorem{notation}[theorem]{Notation}
\newtheorem{remark}[theorem]{Remark}

\theoremstyle{definition}
\newtheorem{example}[theorem]{Example}
\newtheorem{conjecture}[theorem]{Conjecture}
\usepackage{nccmath}
\usepackage{mathtools}
\usepackage[reftex]{theoremref}
\usepackage{tikz-cd} 
\usepackage[left=15mm,right=15mm,top=28mm,bottom=25mm]{geometry}
\begin{document}


\title{Twisted Derivations in Algebraic Number Fields}

\author{Praveen Manju and Rajendra Kumar Sharma}
\date{}
\maketitle

\begin{center}
\noindent{\small Department of Mathematics, \\Indian Institute of Technology Delhi, \\ Hauz Khas, New Delhi-110016, India$^{1}$}
\end{center}

\footnotetext[1]{{\em E-mail addresses:} \url{praveenmanjuiitd@gmail.com}(Corresponding Author: Praveen Manju), \url{rksharmaiitd@gmail.com}(Rajendra Kumar Sharma).}

\medskip

\begin{abstract}
Let $A$ be a commutative ring with unity and $B = A[\theta]$ be an integral extension of $A$. Assume that $B$ is an integral domain with quotient field $\mathbb{K}$ and $\mathbb{E}$ is the minimal splitting field of $\theta$ over $\mathbb{K}$. Suppose $\sigma, \tau: B \rightarrow \mathbb{E}$ are two different ring homomorphisms that fix $A$ element-wise. In this article, we classify all $A$-linear maps $D: B \rightarrow \mathbb{E}$ which are $(\sigma, \tau)$-derivations. Consequently, we classify all $(\sigma, \tau)$-derivations in certain field extensions, algebraic number fields, and their ring of algebraic integers. For the ring of algebraic integers, $O_{\mathbb{K}} = \mathbb{Z}[\zeta]$ of the cyclotomic number field $\mathbb{K} = \mathbb{Q}(\zeta)$ ($\zeta$ an $n^{\text{th}}$ primitive root of unity), and a pair $(\sigma, \tau)$ of two different $\mathbb{Z}$-algebra endomorphisms of $O_{\mathbb{K}}$, we conjecture (using SageMath) a necessary and sufficient condition for a $(\sigma, \tau)$-derivation $D:O_{\mathbb{K}} \rightarrow O_{\mathbb{K}}$ to be inner. This is done for two different forms of $n$: (i) $n = 2^{r}p$ ($r \in \mathbb{N}$ and $p$ an odd rational prime), and (ii) $n=p^{k}$ ($k \in \mathbb{N} \setminus \{1\}$ and $p$ any rational prime). As an application of our main result on classification of $(\sigma, \tau)$-derivations $D:B \rightarrow \mathbb{E}$ and also the conjectures on inner $(\sigma, \tau)$-derivations of $O_{\mathbb{K}}$, we also conjecture the existence and non-existence of non-zero outer derivations of $O_{\mathbb{K}}$ for the above two forms of $n$, thus answering the twisted derivation problem in $O_{\mathbb{K}}$. Finally, as another application of our main result on the classification of $(\sigma, \tau)$-derivations $D:B \rightarrow \mathbb{E}$, we construct some binary Hom-IDD codes in coding theory.
\end{abstract}



\section{Introduction}\label{section 1}
Derivations play an important role in mathematics and physics. The theory of derivations has been developed in rings and various algebras and helps study and understand their structure. For example, BCI-algebras \cite{Muhiuddin2012}, von Neumann algebras \cite{Brear1992}, incline algebras that have many applications \cite{Kim2014}, MV-algebras \cite{KamaliArdakani2013, Mustafa2013}, Banach algebras \cite{Raza2016}, lattices that are very important in fields such as information theory: information recovery, information access management, and cryptanalysis \cite{Chaudhry2011}. $\mathbb{C}^{*}$-algebras, operator algebras, differentiable manifolds, and representation theory of Lie groups are being studied using derivations \cite{Klimek2021}. For a historical account and further applications of derivations, we refer the reader to \cite{Atteya2019, Haetinger2011, MohammadAshraf2006}. Jacobson \cite{Jacobson1964} introduced the idea of an $(s_{1}, s_{2})$-derivation which were later commonly called the $(\sigma, \tau)$ or $(\theta, \phi)$-derivations. These twisted derivations have been studied extensively in prime and semiprime rings and have been mainly used in solving functional equations \cite{Brear1992}. Twisted derivations have huge applications. They are used in multiplicative deformations and discretizations of derivatives that have many applications in models of quantum phenomena and the analysis of complex systems and processes. They are extensively investigated in physics and engineering. Using twisted derivations, Lie algebras are generalized to hom-Lie algebras, and the central extension theory is developed for hom-Lie algebras analogously to that for Lie algebras. Just as Lie algebras were initially studied as algebras of derivations, hom-Lie algebras were defined as algebras of twisted derivations. The generalizations (deformations and analogs) of the Witt algebra, the complex Lie algebra of derivations on the algebra of Laurent polynomials $\mathbb{C}[t, t^{-1}]$ in one variable, are obtained using twisted derivations. Deformed Witt and Virasoro-type algebras have applications in analysis, numerical mathematics, algebraic geometry, arithmetic geometry, number theory, and physics. We refer the reader to \cite{Hartwig2006, Ilwale2023, Larsson2017, Siebert1996} for details. Twisted derivations have been used to generalize Galois theory over division rings and in the study of $q$-difference operators in number theory. The notion of absolute derivations has also been used in number theory (see \cite{Lagarias2005} and \cite{N.Kurokawa2003}). Derivations, especially $(\sigma, \tau)$-derivations of rings, have various applications in coding theory \cite{Creedon2019, Boucher2014}. For more applications of $(\sigma, \tau)$-derivations, we refer the reader to \cite{AleksandrAlekseev2020}.

The earliest work involving derivations in number fields is seen in A. Weil's paper \cite{Weil1943}, which gave an idea to develop the theory of differents in algebraic number fields on an arithmetical analog of differentiation in function fields. In \cite{Kawada1951} and \cite{Kinohara1952}, the authors use A. Weil's idea and derivations to develop the theory of the relative differents in algebraic number fields. In \cite{Narkiewicz1969}, the author proves a theorem of A. Weil concerning the notion of essential $I$-derivation ($I$ an ideal in the ring of algebraic integers of a finite extension of an algebraic number field). Since then, no literature can be found that explores derivations in number fields, except \cite{Chaudhuri, Manju2023b}, where the authors study twisted derivations in number fields. Twisted derivations in number fields are yet to be explored.
Derivations have been studied, in general, in field extensions as well. The relationship between derivations and field extensions can be found in \cite{Davis1973}. Some more references where derivations in field extensions have been studied are  \cite{Callahan1973, Derksen1993, Deveney1979, Heerema1981, Heerema1962, Messmer1995, Mordeson1972, Suzuki1981, Thwing1974, Ziegler2003}. Also, this article considers the twisted derivation problem: Are all twisted derivations inner? Or is the space of outer twisted derivations trivial? This problem was initially posed for derivations in group algebras. We refer the reader to \cite{AleksandrAlekseev2020, A.A.Arutyunov2020, Arutyunov2020b, Arutyunov2020} for the history and importance of the derivation problem. This article considers the analogous problem for $(\sigma, \tau)$-derivations in the ring of algebraic integers of a cyclotomic number field.

This article focuses on understanding twisted derivations in algebraic number fields. For the basics of algebraic number theory, we refer the reader to \cite{Marcus2018} and \cite{IanStewart2002}. A $(\sigma, \tau)$-derivation $D: \mathcal{A} \rightarrow \mathcal{B}$ satisfies the identity \begin{equation*} D(\alpha^{k}) = \left(\sum_{i+j=k-1} \sigma(\alpha^{i}) \tau(\alpha^{j})\right) D(\alpha)\end{equation*} for all $\alpha \in \mathcal{A}$ and for all $k \in \mathbb{N}$, where $i,j$ run over non-negative integers. A natural question one may ask is: under what assumptions on an $R$-linear map $D:\mathcal{A} \rightarrow \mathcal{B}$ satisfying the above relations is a $(\sigma, \tau)$-derivation? Similar questions have been posed for derivations of a ring (for example, see \cite{Bridges1984, hosseini2018identities, Vukman2005}). We divide the article into seven sections. In Section \ref{section 2}, we include some preliminaries that are the basis of our main works in Section \ref{section 3} and Section \ref{section 4}. The converse of the \th\ref{lemma 2.2} is not generally true. In \cite{Manju2023b}, the author proved that the converse holds in the ring of algebraic integers of quadratic and $p^{\text{th}}$-cyclotomic number fields ($p$ odd rational prime). In Section \ref{section 3}, we prove that the converse holds in some special commutative rings and algebras (see \th\ref{theorem 3.2}). \th\ref{lemma 2.7} is crucial in proving this main theorem. In particular, the converse holds in certain field extensions (\th\ref{corollary 3.3}, \th\ref{corollary 3.4}, and \th\ref{corollary 3.5}) and in the ring of algebraic integers of all monogenic number fields that are normal or Galois extensions of the field $\mathbb{Q}$ of rational numbers (\th\ref{corollary 3.6} and \th\ref{corollary 3.7}). Let $\mathbb{K} = \mathbb{Q}(\zeta)$ denote the $n^{\text{th}}$-cylcotomic number field ($n \geq 3$ is a positive integer), where $\zeta$ is a primitive $n^{\text{th}}$-root of unity. In \cite{Manju2023b}, the authors developed a conjecture for a $(\sigma, \tau)$-derivation $D: O_{\mathbb{K}} \rightarrow O_{\mathbb{K}}$ to be inner in the case when $n$ is an odd rational prime $p$. In Section \ref{section 4}, we develop some conjectures based on \th\ref{lemma 2.6} giving a necessary and sufficient condition for a $(\sigma, \tau)$-derivation $D:O_{\mathbb{K}} \rightarrow O_{\mathbb{K}}$ to be inner for two different forms of $n$: (i) $n = 2^{r}p$ ($r \in \mathbb{N}$ and $p$ an odd rational prime), and (ii) $n=p^{k}$ ($k \in \mathbb{N} \setminus \{1\}$ and $p$ any rational prime). The case: $n=p^{k}$,  where $k = 1$ and $p$ is an odd rational prime, has already been dealt with in \cite{Manju2023b}. The software SageMath is crucial in the development of these conjectures. In Section \ref{section 5}, we conjecture a solution to the twisted derivation problem in the ring of algebraic integers $O_{\mathbb{K}}$ of an $n^{\text{th}}$-cyclotomic number field for those mentioned above two different forms of $n$. In Section \ref{section 6}, we construct some Hom-IDD codes (introduced in \cite{Manju2023b}) using our main Theorem \ref{theorem 3.2}. Finally, in Section \ref{section 7}, we conclude our findings.

\section{Preliminaries}\label{section 2}
In this section, we discuss some common knowledge on $(\sigma, \tau)$-derivations: some basic definitions, remarks, examples, and lemmas.

Let $R$ be a commutative ring with unity, and $\mathcal{A}$ and $\mathcal{B}$ be $R$-algebras with $\mathcal{A} \subseteq \mathcal{B}$. Let $\sigma, \tau: \mathcal{A} \rightarrow \mathcal{B}$ be two $R$-algebra homomorphisms. 

\begin{definition}
A $(\sigma, \tau)$-derivation $D:\mathcal{A} \rightarrow \mathcal{B}$ is an $R$-linear map satisfying the twisted generalized identity: $D(\alpha \beta) = D(\alpha) \tau(\beta) + \sigma(\alpha) D(\beta)$ for all $\alpha, \beta \in \mathcal{A}$. It is called inner if there exists some $\gamma \in \mathcal{A}$ such that $D(\alpha) = \gamma \tau(\alpha) - \sigma(\alpha) \gamma$ for all $\alpha \in \mathcal{A}$. The elements of the quotient of the $R$-module of all $(\sigma, \tau)$-derivations by the $R$-submodule of all inner $(\sigma, \tau)$-derivations are called outer $(\sigma, \tau)$-derivations.
\end{definition}

\begin{definition}
When $\sigma$ and $\tau$ are identity algebra homomorphisms, then a $(\sigma, \tau)$-derivation satisfies the usual Leibniz rule: $D(\alpha \beta) = D(\alpha) \beta + \alpha D(\beta)$ for all $\alpha, \beta \in \mathcal{A}$, and then a $(\sigma, \tau)$-derivation, inner $(\sigma, \tau)$-derivation, and outer $(\sigma, \tau)$-derivation are called derivation, inner derivation, and outer derivation, respectively.
\end{definition}

\begin{notation}
We denote the set of all $(\sigma, \tau)$-derivations $D:\mathcal{A} \rightarrow \mathcal{B}$ by $\mathcal{D}_{(\sigma, \tau)}(\mathcal{A}, \mathcal{B})$, the set of all inner $(\sigma, \tau)$-derivations by $\text{Inn}_{(\sigma, \tau)}(\mathcal{A}, \mathcal{B})$, and the set of all outer $(\sigma, \tau)$-derivations by $\text{Out}_{(\sigma, \tau)}(\mathcal{A}, \mathcal{B})$.  If $\mathcal{B} = \mathcal{A}$, then we denote these sets simply by $\mathcal{D}_{(\sigma, \tau)}(\mathcal{A})$, $\text{Inn}_{(\sigma, \tau)}(\mathcal{A})$, and $\text{Out}_{(\sigma, \tau)}(\mathcal{A})$. 
\end{notation}

\begin{remark}
If $\mathcal{A}$ is unital having unity $1$, the algebra homomorphisms $\sigma, \tau:\mathcal{A} \rightarrow \mathcal{B}$ are also unital, that is, $\sigma(1) = \tau(1) = 1$, and $D:\mathcal{A} \rightarrow \mathcal{B}$ is a $(\sigma, \tau)$-derivation, then $D(1) = 0$. Also, $\mathcal{D}_{(\sigma, \tau)}(\mathcal{A}, \mathcal{B})$ is an $R$- as well as an $\mathcal{A}$-module, and $\text{Inn}_{(\sigma, \tau)}(\mathcal{A}, \mathcal{B})$ forms its submodule.
\end{remark}

\noindent The outer $(\sigma, \tau)$-derivations are precisely the elements of the factor module $\text{Out}_{(\sigma, \tau)}(\mathcal{A}, \mathcal{B})  = \frac{\mathcal{D}_{(\sigma, \tau)}(\mathcal{A}, \mathcal{B})}{\text{Inn}_{(\sigma, \tau)}(\mathcal{A}, \mathcal{B})}$. Also, note that the set $\mathcal{D}_{(\sigma, \tau)}(\mathcal{A}, \mathcal{B}) \setminus \text{Inn}_{(\sigma, \tau)}(\mathcal{A}, \mathcal{B})$ is the set of all non-inner $(\sigma, \tau)$-derivations. The following lemma establishes a connection between our notions of outer $(\sigma, \tau)$-derivations and non-inner $(\sigma, \tau)$-derivations. Its proof is similar to the one given in \cite{Manju2024}.
\begin{lemma}
Let $T = \{D_{i} \in \mathcal{D}_{(\sigma, \tau)}(\mathcal{A}, \mathcal{B}) \mid i \in I\}$ ($I$ some indexing set) be a left transversal of $\text{Inn}_{(\sigma, \tau)}(\mathcal{A}, \mathcal{B})$ in $\mathcal{D}_{(\sigma, \tau)}(\mathcal{A}, \mathcal{B})$ with $0$ as the coset representative of the coset $\text{Inn}_{(\sigma, \tau)}(\mathcal{A}, \mathcal{B})$. Then the non-inner $(\sigma, \tau)$-derivations  correspond to the elements in the set $\bigcup_{D_{i} \in T \setminus \{0\}} (D_{i} + \text{Inn}_{(\sigma, \tau)}(\mathcal{A}, \mathcal{B}))$. More precisely, $\mathcal{D}_{(\sigma, \tau)}(\mathcal{A}, \mathcal{B}) \setminus \text{Inn}_{(\sigma, \tau)}(\mathcal{A}, \mathcal{B}) = \bigcup_{D_{i} \in T \setminus \{0\}} (D_{i} + \text{Inn}_{(\sigma, \tau)}(\mathcal{A}, \mathcal{B}))$.
\end{lemma}

\begin{remark}\label{remark}
In view of the above lemma, $D$ is a non-inner derivation if and only if $D + \text{Inn}_{(\sigma, \tau)}(\mathcal{A}, \mathcal{B})$ is a non-zero element of $\text{Out}_{(\sigma, \tau)}(\mathcal{A}, \mathcal{B}) = \frac{\mathcal{D}_{(\sigma, \tau)}(\mathcal{A}, \mathcal{B})}{\text{Inn}_{(\sigma, \tau)}(\mathcal{A}, \mathcal{B})}$. In other words, $D$ is a non-inner derivation if and only if $D + \text{Inn}_{(\sigma, \tau)}(\mathcal{A}, \mathcal{B})$ is a non-trivial outer derivation. Therefore, studying the non-trivial outer derivations is equivalent to studying the non-inner derivations (see {\cite[Chapter 11]{pierce}} for details).
\end{remark}

Now, let $R$ be a commutative ring with unity $1$, and $\mathcal{A}$ and $\mathcal{E}$ with $\mathcal{A} \subseteq \mathcal{E}$ be commutative unital algebras over $R$. Further, let $\sigma$, $\tau: \mathcal{A} \rightarrow \mathcal{E}$ be two different non-zero unital  $R$-algebra homomorphisms. First, we have an important lemma below.

\begin{lemma}\th\label{lemma 2.1}
Let $\mathcal{A}$ be of finite rank $n$ as an $R$-module and let $\{\alpha_{1}, \alpha_{2}, ..., \alpha_{n}\}$ be an $R$-basis of $\mathcal{A}$. Then an $R$-linear map $D:\mathcal{A} \rightarrow \mathcal{E}$ is a $(\sigma, \tau)$-derivation if and only if $$D(\alpha_{i} \alpha_{j}) = D(\alpha_{i}) \tau(\alpha_{j}) + \sigma(\alpha_{i}) D(\alpha_{j})$$ for every $i, j \in \{1, 2, ..., n\}$.
\end{lemma}

\begin{proof} Let $D:\mathcal{A} \rightarrow \mathcal{E}$ be an $R$-linear map, and $x, y \in \mathcal{A}$. Then $x = \sum_{i=1}^{n} a_{i} \alpha_{i}$ and $y = \sum_{j=1}^{n} b_{j} \alpha_{j}$ for some $a_{i}, b_{j} \in R$ ($i, j \in \{1, 2, ..., n\}$). Observe that $$xy = \left(\sum_{i=1}^{n} a_{i} \alpha_{i}\right) \left(\sum_{j=1}^{n} b_{j} \alpha_{j} \right) = \sum_{i, j} (a_{i}b_{j}) (\alpha_{i} \alpha_{j}),$$ and

\begin{eqnarray*}
D(xy) & = & \sum_{i,j=1}^{n} (a_{i}b_{j}) D(\alpha_{i} \alpha_{j}) = \sum_{i,j=1}^{n} (a_{i}b_{j}) (D(\alpha_{i}) \tau(\alpha_{j}) + \sigma(\alpha_{i}) D(\alpha_{j})) \\ & = & D \left(\sum_{i=1}^{n} a_{i}\alpha_{i}\right) \tau \left(\sum_{j=1}^{n} b_{j} \alpha_{j}\right) + \sigma \left(\sum_{i=1}^{n} a_{i}\alpha_{i}\right) D \left(\sum_{j=1}^{n} b_{j} \alpha_{j}\right) \\ & = & D(x) \tau(y) + \sigma(x) D(y). 
\end{eqnarray*}

Hence $D$ is a $(\sigma, \tau)$-derivation. The converse is straightforward.
\end{proof}

Throughout the manuscript, we assume the following notations: $\mathbb{N}$ and $\mathbb{Z}$ denote the set of natural numbers and integers respectively; $\mathbb{N}_{0} = \mathbb{N} \cup \{0\}$; $\text{gcd}(r,s)$ denotes the greatest common divisor of any two non-zero integers $r$ and $s$; $\phi(n)$ denotes the Euler's phi function of the positive integer $n$; $\Phi_{n}(x)$ denotes the $n^{\text{th}}$ cyclotomic polynomial, $O_{\mathbb{K}}$ denotes the ring of algebraic integers of a number field $\mathbb{K}$, and $[\mathbb{K} : \mathbb{F}]$ denotes the degree of the field extension $\mathbb{K}$ of a field $\mathbb{F}$.
For an integer $k$, define $S_{k}$ as the set containing all ordered pairs $(i,j) \in \mathbb{N}_{0} \times \mathbb{N}_{0}$ for which $i+j = k$. Note that the sets $S_{k} \text{'s}$ are pairwise disjoint with $|S_{k}| = k+1$ (cardinality of $S_{k}$) for $k \in \mathbb{N}_{0}$ and $S_{k}$ is an empty set for $k$ a negative integer. Therefore, $|\cup_{i=0}^{m} S_{i}| = \sum_{i=0}^{m} |S_{i}| = 1 + 2 + 3 + ... + m + (m+1) = \frac{(m+1)(m+2)}{2}$ for every non-negative integer $m$. We adopt the convention that empty sums are zero.

\begin{lemma}\th\label{lemma 2.2}
Let $\mathcal{A}$ be of finite rank $n$, and suppose that $\mathcal{A}$ has an $R$-basis of the form $\{1, \alpha, \alpha^{2}, ..., \alpha^{n-1}\}$ for some $\alpha \in \mathcal{A}$. If an $R$-linear map $D:\mathcal{A} \rightarrow \mathcal{E}$ is a $(\sigma, \tau)$-derivation, then \begin{equation}D(\alpha^{k}) = \left(\sum_{(i,j) \in S_{k-1}}  \sigma(\alpha^{i}) \tau(\alpha^{j})\right)D(\alpha)\end{equation} for all $k \in \{1, 2, ..., n-1\}$.
\end{lemma}
\begin{proof} We use induction on $k$. For $k = 1$, the equality trivially holds. Let the result hold for $k=n$, that is, $D(\alpha^{n}) = \left(\sum_{(i,j) \in S_{n-1}}  \sigma(\alpha^{i}) \tau(\alpha^{j})\right)D(\alpha)$. Then

\begin{eqnarray*}
D(\alpha^{n+1}) = D(\alpha^{n} \alpha) & = & D(\alpha^{n}) \tau(\alpha) + \sigma(\alpha^{n}) D(\alpha) \\ & = & \left(\sum_{i+j = n-1} \sigma(\alpha^{i}) \tau(\alpha^{j+1}) + \sigma(\alpha^{n}) \tau(\alpha^{0})\right) D(\alpha) \\ & = & \left(\sum_{i+j = n} \sigma(\alpha^{i}) \tau(\alpha^{j})\right)D(\alpha) = \left(\sum_{(i,j) \in S_{n}}  \sigma(\alpha^{i}) \tau(\alpha^{j})\right)D(\alpha).
\end{eqnarray*}

Induction is complete, as is proof.
\end{proof}

The converse of \th\ref{lemma 2.2} is not generally true, as illustrated in the following examples.

\begin{example}\th\label{example 2.3}
For some positive integer $n>1$, suppose that $A$ is an $n \times n$ nilpotent matrix with entries from the ring of integers $\mathbb{Z}$. Further, suppose that $r$ is the least positive integer such that $A^{r} = 0$. The set $$\mathcal{A} = \{\sum_{i=0}^{r-1} a_{i} A^{i} \mid a_{i} \in \mathbb{Z}, \hspace{0.1cm} \forall \hspace{0.1cm} i \in \{0, 1, ..., r-1\}\}$$ forms a commutative $\mathbb{Z}$-algebra with unity as the identity matrix $I$. Furthermore, it is a $\mathbb{Z}$-module with a $\mathbb{Z}$-basis $\{A^{0}=I, A, ..., A^{r-1}\}$. Define $\sigma, \tau:\mathcal{A} \rightarrow \mathcal{A}$ by $\sigma(A) = A$ and $\tau(A) = B$, where $B = m A$ and $m \in \mathbb{N} \setminus \{1\}$. Extending $\sigma, \tau$ $\mathbb{Z}$-linearly to the whole of $\mathcal{A}$ makes them $\mathbb{Z}$-algebra unital endomorphisms of $\mathcal{A}$. Now, define a $\mathbb{Z}$-linear map $D:\mathcal{A} \rightarrow \mathcal{A}$ by $D(I) = 0$ and $$D(A^{k}) = \left(\sum_{(i,j) \in S_{k-1}}  \sigma(A^{i}) \tau(A^{j})\right)D(A), \hspace{0.8cm} \text{for} \hspace{0.2cm} k \in \{1, ..., r-1\},$$ and extending it $\mathbb{Z}$-linearly to the whole of $\mathcal{A}$.

Then $D(A^{r}) = 0$. And, \begin{eqnarray*}\left(\sum_{(i,j) \in S_{r-1}} \sigma(A^{i}) \tau(A^{j})\right)D(A) = \left(\sum_{i+j=r-1} A^{i} B^{j}\right)D(A) & = & \left(\sum_{j=0}^{r-1} m^{j} A^{r-1} \right)D(A) \\ & = & \left(\frac{1-m^{r}}{1-m}\right)A^{r-1}D(A).\end{eqnarray*}

Choose $D(A) \in \mathcal{A}$ such that $A^{r-1} D(A) \neq 0$ (one such choice for $D(A)$ is $D(A) = sI$, where $s$ is a non-zero integer). Then  $0 \neq \left(\frac{1-m^{r}}{1-m}\right)A^{r-1}D(A)$. Therefore, $D$ is not a $(\sigma, \tau)$-derivation of $\mathcal{A}$.
\end{example}

\begin{example}\th\label{example 2.4}
Consider a commutative unital ring $R$ and the polynomial ring $R[X]$ in the variable $X$. Let $f(X) \in R[X]$ be a monic polynomial of degree $m$. Then $\mathcal{A} = \frac{R[X]}{\langle f(X) \rangle}$ is a commutative $R$-algebra with unity $\bar{1} = 1 + \langle f(X) \rangle$ and the zero element $\bar{0} = 0 + \langle f(X) \rangle$. The set $\{1, \alpha, \alpha^{2}, ..., \alpha^{m-1}\}$ forms an $R$-basis of the $R$-module $\mathcal{A}$, where $\alpha = X + \langle f(X) \rangle$. 

In particular, we take $R = \mathbb{Z}$ and $f(X) = X^{6} - 1$ so that $\{\bar{1}, \alpha, \alpha^{2}, \alpha^{3}, \alpha^{4}, \alpha^{5}\}$ becomes a $\mathbb{Z}$-basis of $\mathcal{A} = \frac{\mathbb{Z}[X]}{\langle f(X) \rangle}$. Define $\sigma, \tau:\mathcal{A} \rightarrow \mathcal{A}$ by $\sigma(\alpha) = \alpha$ and $\tau(\alpha) = \alpha^{2}$. Then $\sigma, \tau$ become $\mathbb{Z}$-algebra unital endomorphisms of $\mathcal{A}$. Define a $\mathbb{Z}$-linear map $D:\mathcal{A} \rightarrow \mathcal{A}$ by $D(\bar{1}) = \bar{0}$ and $$D(\alpha^{k}) = \left(\sum_{(i,j) \in S_{k-1}}  \sigma(\alpha^{i}) \tau(\alpha^{j})\right)D(\alpha)$$ for all $k \in \{1, 2, 3, 4, 5\}$.

$\alpha^{6} = \bar{1}$ implies that $D(\alpha^{6}) = \bar{0}$. And, 
\begin{equation*}
\begin{aligned}
\left(\sum_{(i,j) \in S_{5}}  \sigma(\alpha^{i}) \tau(\alpha^{j})\right)D(\alpha) & = (\sigma(\alpha^{5}) + \sigma(\alpha^{4}) \tau(\alpha) + \sigma(\alpha^{3}) \tau(\alpha^{2}) + \sigma(\alpha^{2}) \tau(\alpha^{3}) + \sigma(\alpha)\tau(\alpha^{4}) \\ &\quad + \tau(\alpha^{5}))D(\alpha) 
\\ & = \left(\alpha^{5} + \alpha^{6} + \alpha^{7}  + \alpha^{8} + \alpha^{9} + \alpha^{10} \right)D(\alpha) 
\\ & = \left(\bar{1} + \alpha + \alpha^{2} + \alpha^{3} + \alpha^{4} + \alpha^{5}\right)D(\alpha).
\end{aligned}
\end{equation*}
As $\{\bar{1}, \alpha, \alpha^{2}, \alpha^{3}, \alpha^{4}, \alpha^{5}\}$ is a $\mathbb{Z}$-basis of $\mathcal{A}$, therefore, $\bar{1} + \alpha + \alpha^{2} + \alpha^{3} + \alpha^{4} + \alpha^{5} \neq 0$. Suppose that $D(\alpha)$ is chosen in $\mathcal{A}$ such that the above expression is never $0$ (one such choice for $D(\alpha)$ is $D(\alpha) = \alpha$). Then, $D(\alpha^{4}) \neq \left(\sum_{(i,j) \in S_{5}}  \sigma(\alpha^{i}) \tau(\alpha^{j})\right)D(\alpha)$. Therefore, $D$ is not a $(\sigma, \tau)$-derivation in view of the \th\ref{lemma 2.2}.
\end{example}

\begin{lemma}\th\label{lemma 2.5}
Let $\mathcal{A}$ be of finite rank $n$ and $\{\alpha_{1}, \alpha_{2}, ..., \alpha_{n}\}$ be an $R$-basis of $\mathcal{A}$. Then a $(\sigma, \tau)$-derivation $D:\mathcal{A} \rightarrow \mathcal{A}$ is inner if and only if there exists some $\beta \in \mathcal{A}$ such that $D(\alpha_{i}) = \beta (\tau - \sigma)(\alpha_{i})$ for all $i \in \{1, 2, ..., n\}$.
\end{lemma}
\begin{proof}
The forward part follows directly from the definition of an inner $(\sigma, \tau)$-derivation. For the converse, note that for $\alpha = \sum_{i=1}^{n} a_{i} \alpha_{i} \in \mathcal{A}$, \begin{eqnarray*}D(\alpha) = \sum_{i=1}^{n} a_{i} D(\alpha_{i}) = \sum_{i=1}^{n} a_{i} \beta (\tau - \sigma)(\alpha_{i}) & = & \beta \left(\sum_{i=1}^{n} a_{i} \tau(\alpha_{i}) - \sum_{i=1}^{n} a_{i} \sigma(\alpha_{i})\right) \\ & = & \beta \left(\tau\left(\sum_{i=1}^{n} a_{i} \alpha_{i}\right) - \sigma\left(\sum_{i=1}^{n} a_{i} \alpha_{i}\right)\right) \\ & = & \beta (\tau - \sigma)(\alpha)\end{eqnarray*}

Since $\alpha \in \mathcal{A}$ is arbitrary, therefore, $D$ is inner. 
\end{proof}

\begin{lemma}\th\label{lemma 2.6}
Let $\mathcal{A}$ be of finite rank $n$, and suppose that $\mathcal{A}$ has an $R$-basis of the form $\{1, \alpha, \alpha^{2}, ..., \alpha^{n-1}\}$ for some $\alpha \in \mathcal{A}$. Then a $(\sigma, \tau)$-derivation $D:\mathcal{A} \rightarrow \mathcal{A}$ is inner if and only if there exists some $\beta \in \mathcal{A}$ such that $D(\alpha) = \beta (\tau - \sigma)(\alpha)$.
\end{lemma}
\begin{proof}
The forward implication follows directly from the definition. For the converse, first note by induction that $D(\alpha^{k}) = \beta (\tau - \sigma)(\alpha^{k})$ for all $k \in \mathbb{N}$. So, in particular, this holds for all $k \in \{1, ..., n-1\}$. Also, $D(1) = 0$ and $\sigma(1) = \tau(1) = 1$, so $D(1) = \beta (\tau - \sigma)(1)$. Now, the proof follows using \th\ref{lemma 2.5}.
\end{proof}

Below, we have a useful lemma.

\begin{lemma}\th\label{lemma 2.7}
Let $\mathbb{K}$ be a field, and $a, r \in \mathbb{K}$. Consider the geometric sequence whose $n^{\text{th}}$ term is $a_{1}r^{n-1}$. The sum of the first $n$ terms of the geometric sequence $a_{1}, a_{1}r, ..., a_{1}r^{n-1}$ is $na_{1}$ if $r=1$ and $a_{1}(1-r)^{-1}(1-r^{n})$ if $r \neq 1$.
\end{lemma}
\begin{proof}
Let $S = a_{1} + a_{1}r + ... + a_{1}r^{n-1}$. Compute $rS$. Then $(1-r)S = a_{1}(1-r^{n})$ so that $S = a_{1}(1-r)^{-1}(1-r^{n})$.
\end{proof}

Below, we list some elementary properties of the cyclotomic polynomials that were helpful in our calculations. We refer the reader to \cite{Ge2008} and \cite{Porter2015} for the proofs and further details.

\begin{lemma}\th\label{lemma 2.8}
Let $n, m, k$ be positive integers and $p$ be a rational prime.
\begin{enumerate}
\item[(i)] If $n = \prod_{i=1}^{s} p_{i}^{k_{i}}$ is the prime factorization of $n$ and $m = \prod_{i=1}^{s} p_{i}^{l_{i}}$ with $1 \leq l_{i} \leq k_{i}$ for all $1 \leq i \leq s$, then $\Phi_{n}(x) = \Phi(x^{\frac{n}{m}})$.
\item[(ii)] $\Phi_{p}(x) = \sum_{i=0}^{p-1} x^{i}$.
\item[(iii)] $\Phi_{p^{k}}(x) = \Phi_{p}(x^{p^{k-1}})$.
\item[(iv)] $\Phi_{p^{k}n}(x) = \Phi_{n}(x^{p^{k}})$ if $p$ divides $n$ and $\Phi_{p^{k}n}(x) = \frac{\Phi_{n}(x^{p^{k}})}{\Phi_{n}(x^{p^{k-1}})}$ if $p$ does not divide $n$.
\item[(v)] If $n > 1$ is odd, then $\Phi_{2n}(x) = \Phi_{n}(-x)$.
\item[(vi)] If $p$ is odd, then $\Phi_{2^{k}p}(x) = \Phi_{p}(-x^{2^{k-1}})$.
\end{enumerate}
\end{lemma}

\section{$(\sigma, \tau)$-Derivations of Integral Extensions}\label{section 3}
Let $A, B$ be commutative rings with unity such that $A \subseteq B$ and $B = A[\theta]$, a polynomial ring in $\theta \in B$ over $A$, where $\theta$ is integral over $A$. Then $\theta$ satisfies a monic polynomial $f(x) = x^{n} + a_{n-1}x^{n-1} + ... + a_{1}x + a_{0}$ for some $a_{i} \in A$ ($0 \leq i \leq n-1$) such that $f(x)$ divides every other polynomial that is satisfied by $\theta$. Put $a_{n} = 1$ so that $f(x) = \sum_{i=0}^{n} a_{i}x^{i}$. So $B$ becomes an integral extension of $A$. The set $\{1, \theta, ..., \theta^{n-1}\}$ is a basis of $B$ over $A$, and $B$ becomes an $A$-module of rank $n$. In addition, suppose that $B$ is an integral domain with quotient field $\mathbb{K}$ and $\mathbb{E}$ is the minimal splitting field of $\theta$ over $\mathbb{K}$. Let $\sigma, \tau: B \rightarrow \mathbb{E}$ be two different ring homomorphisms that fix $A$ element-wise. Equivalently, let $\sigma, \tau:B \rightarrow \mathbb{E}$ be two different $A$-algebra homomorphisms. Then $\sigma(\theta), \tau(\theta)$ also become zeros of $f(x)$.

The following lemma is crucial in proving this section's main \th\ref{theorem 3.2}.
\begin{lemma}\th\label{lemma 3.1}
For any $k \in \mathbb{N}_{0}$, $$\sum_{i=k}^{n+k} a_{i-k} \left(\sum_{(s,t) \in S_{i-1}} (\sigma(\theta))^{s} (\tau(\theta))^{t} \right) = 0.$$
\end{lemma}
\begin{proof}
If $\alpha \in \mathbb{E}$ is a zero of $f(x)$, then $f(\alpha) = 0$ so that $\alpha^{n} = - \sum_{i=0}^{n-1} a_{i} \alpha^{i}$.

So for any $k \in \mathbb{N}$, $\alpha^{n+k} = 
- \sum_{i=0}^{n-1} a_{i} \alpha^{i + k} = - \sum_{i=k}^{n-1+k} a_{i-k} \alpha^{i}$ so that $\sum_{i=k}^{n+k} a_{i-k} \alpha^{i} = 0$. In particular, $\sum_{i=k}^{n+k} a_{i-k} (\sigma(\theta))^{i} = 0$ and $\sum_{i=k}^{n+k} a_{i-k} (\tau(\theta))^{i} = 0$ as $\sigma(\theta)$ and $\tau(\theta)$ are zeros of $f(x)$.

Now, using the above observation and \th\ref{lemma 2.7}, we get the following:
\begin{equation*}
\begin{aligned}
\sum_{i=k}^{n+k} a_{i-k} \left(\sum_{(s,t) \in S_{i-1}} (\sigma(\theta))^{s} (\tau(\theta))^{t} \right) & = \sum_{i=k}^{n+k} a_{i-k} \left(\sum_{s+t=i-1} (\sigma(\theta))^{s} (\tau(\theta))^{t} \right) 
\\ & = \sum_{i=k}^{n+k} a_{i-k} \left(\sum_{t=0}^{i-1} (\sigma(\theta))^{i-1-t} (\tau(\theta))^{t} \right)
\\ & = \sum_{i=k}^{n+k} a_{i-k} (\sigma(\theta))^{i-1} \left(\sum_{t=0}^{i-1} (\sigma(\theta^{-1}) \tau(\theta))^{t} \right)
\\ & = (1-\sigma(\theta^{-1})\tau(\theta))^{-1} \sum_{i=k}^{n+k} a_{i-k} (\sigma(\theta))^{i-1} (1-(\sigma(\theta^{-1}) \tau(\theta))^{i})
\\ & = (1-\sigma(\theta^{-1})\tau(\theta))^{-1}(\sigma(\theta))^{-1} \sum_{i=k}^{n+k} a_{i-k} ((\sigma(\theta))^{i} - (\tau(\theta))^{i})
\\ & = (\sigma(\theta) - \tau(\theta))^{-1} \left(\sum_{i=k}^{n+k} a_{i-k} (\sigma(\theta))^{i} - \sum_{i=t}^{n+t} a_{i} (\tau(\theta))^{i}\right)
\\ & = 0
\end{aligned}
\end{equation*}
\end{proof}

\begin{theorem}\th\label{theorem 3.2}
Let $D:B \rightarrow \mathbb{E}$ be an $A$-linear map with $D(1) = 0$ and satisfying the $n-1$ equations \begin{equation}\label{eq 3.1}
D(\theta^{i}) = \left( \sum_{(s,t) \in S_{i-1}} (\sigma(\theta))^{s} (\tau(\theta))^{t} \right) D(\theta),
\end{equation} for all $i \in \{1, ..., n-1\}$. Then $D$ is a $(\sigma, \tau)$-derivation.
\end{theorem}
\begin{proof}
In view of \th\ref{lemma 2.1}, it is enough to prove that \begin{equation}\label{eq 3.2}
D(\theta^{i}\theta^{j}) = D(\theta^{i})\tau(\theta^{j}) + \sigma(\theta^{i})D(\theta^{j})\end{equation} for all $i, j \in \{0, 1, ..., n-1\}$.

The equality holds trivially if $i$ or $j$ is $0$. So assume that $i, j \in \{1, ..., n-1\}$. Using (\ref{eq 3.1}), we get that

\begin{eqnarray*}
D(\theta^{i}) \tau(\theta^{j}) = \left(\sum_{(s,t) \in S_{i-1}}  \sigma(\theta^{s}) \tau(\theta^{t})\right)D(\theta) \tau(\theta^{j}) & = & \left(\sum_{(s,t) \in S_{i-1}}  \sigma(\theta^{s}) \tau(\theta^{t+j})\right)D(\theta) \\ & = & \left(\sum_{s=0}^{i-1} \sigma(\theta^{s}) \tau(\theta^{i-1-s+j})\right)D(\theta)\end{eqnarray*} and \begin{eqnarray*}
\sigma(\theta^{i}) D(\theta^{j}) = \sigma(\theta^{i})\left(\sum_{(s,t) \in S_{j-1}} \sigma(\theta^{s}) \tau(\theta^{t})\right)D(\theta) & = & \left(\sum_{(s,t) \in S_{j-1}} \sigma(\theta^{i+s}) \tau(\theta^{t})\right)D(\theta) \\ & = & \left(\sum_{s=0}^{j-1} \sigma(\theta^{i+s}) \tau(\theta^{j-1-s})\right)D(\theta) \\ & = & \left(\sum_{s=i}^{i+j-1} \sigma(\theta^{s}) \tau(\theta^{j-1+i-s})\right)D(\theta).
\end{eqnarray*}
Therefore, \begin{equation}\label{eq 3.3}
D(\theta^{i}) \tau(\theta^{j}) + \sigma(\theta^{i}) D(\theta^{j}) = \left(\sum_{(s,t) \in S_{i+j-1}} \sigma(\theta^{s}) \tau(\theta^{t})\right)D(\theta).
\end{equation}

We divide the proof into cases.

\textbf{Case 1: $i+j \leq n-1$.}

Then the relation (\ref{eq 3.2}) holds using (\ref{eq 3.1}).

\textbf{Case 2: $i+j = n$.}

$\theta$ satisfies the polynomial $f(x) = \sum_{i=0}^{n} a_{i}x^{i}$, where $a_{n}=1$. So $\theta^{n} = - \sum_{i=0}^{n-1} a_{i} \theta^{i}$. So using (\ref{eq 3.1}) and the fact that $D(1) = 0$, we get that $$D(\theta^{n}) = - \sum_{i=1}^{n-1} a_{i} D(\theta^{i}) = - \sum_{i=1}^{n-1} a_{i} \left( \sum_{(s,t) \in S_{i-1}} (\sigma(\theta))^{s} (\tau(\theta))^{t} \right) D(\theta).$$

Also, by (\ref{eq 3.3}), $D(\theta^{i}) \tau(\theta^{j}) + \sigma(\theta^{i}) D(\theta^{j}) = \left(\sum_{(s,t) \in S_{n-1}} \sigma(\theta^{s}) \tau(\theta^{t})\right)D(\theta)$.

So to show that $D(\theta^{i+j}) = D(\theta^{i}) \tau(\theta^{j}) + \sigma(\theta^{i}) D(\theta^{j})$, we need to show that $$\left(\sum_{i=1}^{n} a_{i} \left( \sum_{(s,t) \in S_{i-1}} (\sigma(\theta))^{s} (\tau(\theta))^{t} \right)\right) D(\theta) = 0.$$

Since, by our convention, empty sums are zero. Therefore, we need to show that $$\left(\sum_{i=0}^{n} a_{i} \left( \sum_{(s,t) \in S_{i-1}} (\sigma(\theta))^{s} (\tau(\theta))^{t} \right)\right) D(\theta) = 0.$$

This will be proved once we prove that $\sum_{i=0}^{n} a_{i} \left( \sum_{(s,t) \in S_{i-1}} (\sigma(\theta))^{s} (\tau(\theta))^{t} \right) = 0$. However, this holds by putting $k=0$ in \th\ref{lemma 3.1}.

So, in this case, too, the relation (\ref{eq 3.2}) holds.

$i, j \in \{1, ..., n-1\}$, so $2 \leq i + j \leq 2n-2 = n + (n-2)$. So now, to show that $D(\theta^{i+j}) = D(\theta^{i}) \tau(\theta^{j}) + \sigma(\theta^{i}) D(\theta^{j})$, we use induction.

\textbf{Case 3: $i+j = n+k$, where $k \in \{1, ..., n-2\}$.}

Let $k=1$. Since $\theta^{n+1} = - \sum_{i=1}^{n} a_{i-1} \theta^{i}$ and as proved above, the result is true for $k=0$, therefore, $D(\theta^{n+1}) = - \sum_{i=1}^{n} a_{i-1} D(\theta^{i}) = - \sum_{i=1}^{n} a_{i-1} \left( \sum_{(s,t) \in S_{i-1}} (\sigma(\theta))^{s} (\tau(\theta))^{t} \right)D(\theta)$.

Also, by (\ref{eq 3.3}), $D(\theta^{i}) \tau(\theta^{j}) + \sigma(\theta^{i}) D(\theta^{j}) = \left(\sum_{(s,t) \in S_{n}} \sigma(\theta^{s}) \tau(\theta^{t})\right)D(\theta)$.

Therefore, $D(\theta^{n+1}) = D(\theta^{i}) \tau(\theta^{j}) + \sigma(\theta^{i}) D(\theta^{j})$ if and only if $$\sum_{i=1}^{n+1} a_{i-1} \left( \sum_{(s,t) \in S_{i-1}} (\sigma(\theta))^{s} (\tau(\theta))^{t} \right)D(\theta) = 0$$ which, in turn, holds if $$\sum_{i=1}^{n+1} a_{i-1} \left( \sum_{(s,t) \in S_{i-1}} (\sigma(\theta))^{s} (\tau(\theta))^{t} \right) = 0.$$ But by \th\ref{lemma 3.1}, the above equality holds. Therefore, the result is true for $k=1$.

Now, assume that the result is true for any $k \leq r$ in $\{1, ..., n-3\}$, that is, for a fixed $r \in \{1, ..., n-3\}$,  for any $i, j \in \{1, ..., n-1\}$ with $i+j \leq n+r$, $D(\theta^{i+j}) = D(\theta^{i}) \tau(\theta^{j}) + \sigma(\theta^{i}) D(\theta^{j})$. This gives that for any $l : 1 \leq l \leq n+r$, $D(\theta^{l}) = \left(\sum_{(s,t) \in S_{l-1}} (\sigma(\theta))^{s} (\tau(\theta))^{t}\right)D(\theta)$.

We now prove the result for $k = r+1$. So let $i, j \in \{1, ..., n-1\}$ with $i+j = n+r+1$.

$\theta^{n+r+1} = - \sum_{i=r+1}^{n+r} a_{i-r-1}\theta^{i}$. By the induction hypothesis, the result is true for $k \leq r$; therefore, $$D(\theta^{n+r+1}) = - \sum_{i=r+1}^{n+r} a_{i-r-1} D(\theta^{i}) = - \sum_{i=r+1}^{n+r} a_{i-r-1} \left( \sum_{(s,t) \in S_{i-1}} (\sigma(\theta))^{s} (\tau(\theta))^{t} \right)D(\theta).$$

Also, by (\ref{eq 3.3}), $D(\theta^{i}) \tau(\theta^{j}) + \sigma(\theta^{i}) D(\theta^{j}) = \left(\sum_{(s,t) \in S_{n+r}} \sigma(\theta^{s}) \tau(\theta^{t})\right)D(\theta)$. 

Therefore, $D(\theta^{n+r+1}) = D(\theta^{i}) \tau(\theta^{j}) + \sigma(\theta^{i}) D(\theta^{j})$ if and only if $$\sum_{i=r+1}^{n+r+1} a_{i-r-1} \left( \sum_{(s,t) \in S_{i-1}} (\sigma(\theta))^{s} (\tau(\theta))^{t} \right)D(\theta) = 0$$ which, in turn, holds if $$\sum_{i=r+1}^{n+r+1} a_{i-r-1} \left( \sum_{(s,t) \in S_{i-1}} (\sigma(\theta))^{s} (\tau(\theta))^{t} \right) = 0.$$ But by \th\ref{lemma 3.1}, the above equality holds. Therefore, the result is true for $k=r+1$.

Therefore, the result is valid for all $k \in \{1, ..., n-2\}$.

Therefore, the relation (\ref{eq 3.2}) holds for all $i, j \in \{0, 1, ..., n-1\}$. Hence, by \th\ref{lemma 2.1}, $D$ becomes a $(\sigma, \tau)$-derivation.
\end{proof}





Now, we have the following immediate corollaries.

\begin{corollary}\th\label{corollary 3.3}
Let $\mathbb{K}$ be a finite simple field extension of a field $\mathbb{F}$ so that $\mathbb{K} = \mathbb{F}(\theta)$ for some $\theta \in \mathbb{K}$ having degree $n$ over $\mathbb{F}$, for some fixed $n \in \mathbb{N}$. Further, let $\mathbb{E}$ be the minimal splitting field of $\theta$ over $\mathbb{K}$. Let $\sigma, \tau:\mathbb{K} \rightarrow \mathbb{E}$ be two different field homomorphisms that fix $\mathbb{F}$ element-wise. Then every $\mathbb{F}$-linear map $D:\mathbb{K} \rightarrow \mathbb{E}$ with $D(1) = 0$ satisfying the $n-1$ equations (\ref{eq 3.1}) is a $(\sigma, \tau)$-derivation. Therefore, the $\mathbb{F}$-module $\mathcal{D}_{(\sigma, \tau)}(\mathbb{K}, \mathbb{E})$ is finitely generated of rank $n$.
\end{corollary}

\begin{corollary}\th\label{corollary 3.4}
Let $\mathbb{K} = \mathbb{Q}(\theta)$, $\theta$ an algebraic integer, be an algebraic number field that is a normal extension of $\mathbb{Q}$. Let $\sigma, \tau:\mathbb{K} \rightarrow \mathbb{K}$ be two different field automorphisms or field homomorphisms that fix $\mathbb{Q}$ element-wise. Then every $\mathbb{Q}$-linear map $D:\mathbb{K} \rightarrow \mathbb{K}$ with $D(1) = 0$ and satisfying the $n-1$ equations (\ref{eq 3.1}) is a $(\sigma, \tau)$-derivation. Therefore, the $\mathbb{Q}$-module $\mathcal{D}_{(\sigma, \tau)}(\mathbb{K})$ is finitely generated of rank $[\mathbb{K}:\mathbb{Q}]$.
\end{corollary}

\begin{corollary}\th\label{corollary 3.5}
Let $\mathbb{K} = \mathbb{Q}(\zeta)$, be an $n^{\text{th}}$-cyclotomic number field. Let $\sigma, \tau:\mathbb{K} \rightarrow \mathbb{K}$ be two different field automorphisms. Then every $\mathbb{Q}$-linear map $D:\mathbb{K} \rightarrow \mathbb{K}$ with $D(1) = 0$ and satisfying the $n-1$ equations (\ref{eq 3.1}) is a $(\sigma, \tau)$-derivation. Therefore, the $\mathbb{Q}$-module $\mathcal{D}_{(\sigma, \tau)}(\mathbb{K})$ is finitely generated of rank $\phi(n)$.
\end{corollary}

Note that the ring of algebraic integers of an algebraic number field can be considered a $\mathbb{Z}$-algebra.

\begin{corollary}\th\label{corollary 3.6}
Let $\mathbb{K}$ be a monogenic number field that is a normal extension of $\mathbb{Q}$, and let $\sigma, \tau: O_{\mathbb{K}} \rightarrow O_{\mathbb{K}}$ be two different $\mathbb{Z}$-algebra endomorphisms. Then every $\mathbb{Z}$-linear map $D:O_{\mathbb{K}} \rightarrow O_{\mathbb{K}}$ with $D(1) = 0$ and satisfying the $n-1$ equations (\ref{eq 3.1}) is a $(\sigma, \tau)$-derivation. Therefore, the $\mathbb{Z}$-module $\mathcal{D}_{(\sigma, \tau)}(O_{\mathbb{K}})$ is finitely generated of rank $[\mathbb{K}:\mathbb{Q}]$, which is also the rank of the free abelian (additive) group $O_{\mathbb{K}}$.
\end{corollary}

\begin{corollary}\th\label{corollary 3.7}
Let $\mathbb{K}$ be an $n^{\text{th}}$-cyclotomic number field and $\sigma, \tau:O_{\mathbb{K}} \rightarrow O_{\mathbb{K}}$ be two different $\mathbb{Z}$-algebra endomorphisms. Then every $\mathbb{Z}$-linear map $D:O_{\mathbb{K}} \rightarrow O_{\mathbb{K}}$ with $D(1) = 0$ and satisfying the $n-1$ equations (\ref{eq 3.1}) is a $(\sigma, \tau)$-derivation. Therefore, the $\mathbb{Z}$-module $\mathcal{D}_{(\sigma, \tau)}(O_{\mathbb{K}})$ is finitely generated of rank $\phi(n)$, which is also the rank of the free abelian (additive) group $O_{\mathbb{K}}$.
\end{corollary}

\section{Inner $(\sigma, \tau)$-Derivations of the Ring of Algebraic Integers of Cyclotomic Number Fields}\label{section 4}
Let $\mathbb{K} = \mathbb{Q}(\zeta)$ be an $n^{\text{th}}$-cyclotomic number field. Then $O_{\mathbb{K}} = \mathbb{Z}[\zeta]$ is its ring of algebraic integers. Let $\sigma, \tau:O_{\mathbb{K}} \rightarrow O_{\mathbb{K}}$ be two different non-zero $\mathbb{Z}$-algebra endomorphisms. Then $\sigma(\zeta) = \zeta^{u}$ and $\tau(\zeta) = \zeta^{v}$ for some $u, v \in U(n)$ with $u \neq v$, where $U(n) = \{x \in \mathbb{N} \mid x < n ~ \& ~ \text{gcd}(x,n) = 1\}$. In this section, we conjecture a characterization for a $(\sigma, \tau)$-derivation $D:O_{\mathbb{K}} \rightarrow O_{\mathbb{K}}$ to be inner. We do that for two different forms of $n$, namely, $n=2^{r}p$ ($r \in \mathbb{N}$) for an odd rational prime $p$ and $n = p^{k}$ ($k \in \mathbb{N} \setminus \{1\}$) for any rational prime $p$. Note that $\sigma(\zeta), \tau(\zeta) \in \{\zeta^{x} \mid x \in U(n)\}$ so that there are $\phi(n)(\phi(n)-1)$ distinct possibilities for the pair $(\sigma(\zeta), \tau(\zeta))$. Also, note that if $D:O_{\mathbb{K}} \rightarrow O_{\mathbb{K}}$ is an inner $(\sigma, \tau)$-derivation, then by symmetry, $D$ is also an inner $(\tau, \sigma)$-derivation. Therefore, we just need to consider $\frac{\phi(n)(\phi(n)-1)}{2}$ distinct possibilities for the pair $(\sigma(\zeta), \tau(\zeta))$ from the calculation point of view.

$\mathbb{K}$ has degree $\phi(n)$ over $\mathbb{Q}$, where $\phi$ is Euler's phi function. Let $\beta = \sum_{i=0}^{\phi(n)-1} b_{i} \zeta^{i} \in \mathbb{K}$, and $D: O_{\mathbb{K}} \rightarrow O_{\mathbb{K}}$ be a $(\sigma, \tau)$-derivation and $D(\zeta) = \sum_{i=0}^{\phi(n)-1} c_{i} \zeta^{i} \in O_{\mathbb{K}}$. Put $X^{T} = (b_{0} ~ b_{1} ~ ... ~ b_{\phi(n)-1})$ and $C^{T} = (c_{0} ~ c_{1} ~ ... ~ c_{\phi(n)-1})$. We denote by $\mathbb{Z}^{\phi(n)}$ the set of all $\phi(n) \times 1$ column matrices with entries from $\mathbb{Z}$, and $Adj(A)$ denotes the adjoint of any square matrix $A$ over $\mathbb{Z}$. Note that since $\sigma \neq \tau$, therefore, $D(\zeta) = \beta(\tau - \sigma)(\zeta)$ always has a solution $\beta$ in $\mathbb{K}$. Let $\beta (\tau - \sigma)(\zeta) = \sum_{i=0}^{\phi(n)-1} \left(\sum_{j=0}^{\phi(n)-1} a_{ij}b_{j}\right)\zeta^{i}$ and $A = [a_{ij}]$ be the $\phi(n) \times \phi(n)$ matrix. Then $D(\zeta) = \beta(\tau - \sigma)(\zeta)$ has a solution, $\beta$ in $O_{\mathbb{K}}$ if and only if $AX = C$ has a solution in $\mathbb{Z}^{\phi(n)}$.

\subsection{For $n=2^{r}p$, $r \in \mathbb{N}$ and $p$ an odd rational prime}\label{subsection 4.1}
In this subsection, we conjecture a characterization for a $(\sigma, \tau)$-derivation $D:O_{\mathbb{K}} \rightarrow O_{\mathbb{K}}$ to be inner when $n$ is of the form $n = 2^{r}p$, where $r \in \mathbb{N}$ and $p$ is an odd rational prime. Then the degree of $\mathbb{K}$ over $\mathbb{Q}$ is $\phi(n) = \phi(2^{r}p) = 2^{r-1}(p-1)$.

The following \th\ref{conjecture 4.1} gives the determinant of the coefficient matrix formed in a particular way.

\begin{conjecture}\th\label{conjecture 4.1}
Let $n=2^{r}p$ ($r \in \mathbb{N}$ and $p$ is an odd rational prime). Suppose $\beta = \sum_{i=0}^{2^{r-1}(p-1)-1} b_{i} \zeta^{i} \in O_{\mathbb{K}}$ and $\beta (\tau - \sigma) (\zeta) =  \sum_{i=0}^{2^{r-1}(p-1)-1} \left( \sum_{j=0}^{2^{r-1}(p-1)-1} a_{ij} b_{j} \right) \zeta^{i}$. Suppose that $v-u = 2^{e_{1}}p^{e_{2}}m$, where $m$ is a positive integer, $e_{1}, e_{2}$ are non-negative integers, and $2, p$ do not divide $m$. Then $A = [a_{ij}]$ is a $2^{r-1}(p-1) \times 2^{r-1}(p-1)$ matrix with determinant 
\begin{enumerate}
\item[(i)] $\pm 2^{2^{e_{1}}(p-1)}$ if $1 \leq e_{1} \leq r-1$ and $e_{2} \geq 1$.
\item[(ii)] $\pm p^{2^{r-1}}$ if $e_{1} \geq r$ and $e_{2} = 0$.
\item[(iii)] $\pm 1$ otherwise.
\end{enumerate}
\end{conjecture}

The sign of the determinant of the matrix $A$ depends on the order in which the rows of $A$ are formed. As a consequence, we propose another \th\ref{conjecture 4.2} which gives a necessary and sufficient condition for a $(\sigma, \tau)$-derivation $D: O_{\mathbb{K}} \rightarrow O_{\mathbb{K}}$ to be inner.

\begin{conjecture}\th\label{conjecture 4.2}
Under the hypotheses of \th\ref{conjecture 4.1}, the following statements hold for a $(\sigma, \tau)$-derivation $D:O_{\mathbb{K}} \rightarrow O_{\mathbb{K}}$ with $D(\zeta) = \sum_{i=0}^{2^{r-1}(p-1)-1} c_{i} \zeta^{i} \in O_{\mathbb{K}}$.
\begin{enumerate}
\item[(i)] If $1 \leq e_{1} \leq r-1$ and $e_{2} \geq 1$, then $D$ is inner if and only if $\frac{1}{2^{2^{e_{1}}(p-1)}}(Adj(A)C) \in \mathbb{Z}^{2^{r}(p-1)}$. In particular, if $2^{2^{e_{1}}(p-1)}$ divides $c_{i}$ for each $i \in \{0, 1, ..., 2^{r-1}(p-1)-1\}$, then $D$ is inner. In fact, if $2$ divides $c_{i}$ for each $i \in \{0, 1, ..., 2^{r-1}(p-1)-1\}$, then $D$ is inner.
\item[(ii)] If $e_{1} \geq r$ and $e_{2} = 0$, then $D$ is inner if and only if $\frac{1}{p^{2^{r-1}}}(Adj(A)C) \in \mathbb{Z}^{2^{r}(p-1)}$. In particular, if $p^{2^{r-1}}$ divides $c_{i}$ for each $i \in \{0, 1, ..., 2^{r-1}(p-1)-1\}$, then $D$ is inner. In fact, if $p$ divides $c_{i}$ for each $i \in \{0, 1, ..., 2^{r-1}(p-1)-1\}$, then $D$ is inner.
\item[(iii)] In all other cases, $D$ is always inner.
\end{enumerate}
\end{conjecture} 

We have verified the \th\ref{conjecture 4.1} and \th\ref{conjecture 4.2} using SageMath for various forms of $n$ as given below. In each case, explicit matrix $A$, $\text{det}(A)$ (determinant of $A$), and the last column $C'$ of the row reduced echelon form $(I ~ C')$ of the augmented matrix $(A ~ C)$ were found for all possible values of the pair $(\sigma(\zeta), \tau(\zeta))$. Nevertheless, we believe the conjectures to be true for all $r \in \mathbb{N}$ and for all odd rational primes $p$. 
\begin{enumerate}
\item[(i)] When $n=2p$ and $p=3, 5, 7$, that is, for the pairs $(r,p) = (1,3), (1,5), (1,7)$; in this article, we record our observations for the primes $p=3$ and $p=5$ in the tables \ref{table 1} and \ref{table 2}.

\item[(ii)] When $n=4p$ and $p=3, 5, 7$, that is, for the pairs $(r,p) = (2,3), (2,5), (2,7)$; we record our observations for the prime $p=3$ in the table \ref{table 2}. 

\item[(iii)] When $n=8p$ and $p=3, 5$, that is, for the pairs $(r,p) = (3,3), (3,5)$; we record our observations for the prime $p=3$ in the table \ref{table 2}. 

\item[(iv)] When $n=16p$ and $p=3$, that is, for the pair $(r,p) = (4,3)$.
\end{enumerate}

As $r$ and $p$ increase, the size of the matrix $A$ ($\phi(n) \times \phi(n)$, where $n=2^{r}p$) increases, and it becomes computationally very tedious for us to deal with matrices of such enormous size.

\begin{longtable}{|c|c|c|}
\caption{\textbf{For} $\bm{(r,p) = (1,5),}$ $\bm{n = 2^{r}p = 10,}$ $\textbf{det}\bm{(A) = 5,}$ $\bm{\phi(10) = |\{1, 3, 7, 9\}| = 4}$ \textbf{and} $\bm{\Phi_{10}(x) = x^{4} - x^{3} + x^{2} - x + 1.$}}
\label{table 1} \\
\hline 
$(\sigma(\zeta), \tau(\zeta))$ & $(A ~ C)$ & $C'$ \\ 
\hline 
$(\zeta, \zeta^{3})$ & $\begin{pmatrix}
0 & 1 & 1 & -1 & c_{0} \\
1 & -1 & 0 & 2 & c_{1} \\
0 & 2 & 0 & -1 & c_{2} \\
-1 & -1 & 1 & 1 & c_{3}
\end{pmatrix}$ & $\frac{1}{5} \begin{pmatrix}
3 c_{0} + 2 c_{1} - 2 c_{2} - 3 c_{3} \\
- c_{0} + c_{1} + 4 c_{2} + c_{3} \\
4 c_{0} + c_{1} - c_{2} + c_{3} \\
-2 c_{0} + c_{1} + 3 c_{2} + 2 c_{3}
\end{pmatrix}$ \\ 
\hline 
$(\zeta, \zeta^{7})$ & $\begin{pmatrix}
0 & 0 & -1 & -2 & c_{0} \\
1 & 0 & 1 & 1 & c_{1} \\
1 & 1 & -1 & -1 & c_{2} \\
0 & 1 & 2 & 1 & c_{3} \\
\end{pmatrix}$ & $\frac{1}{5} \begin{pmatrix}
c_{0} + 4 c_{1} + c_{2} - c_{3} \\
-2 c_{0} - 3 c_{1} + 3 c_{2} + 2 c_{3} \\
3 c_{0} + 2 c_{1} - 2 c_{2} + 2 c_{3} \\
-4 c_{0} - c_{1} + c_{2} - c_{3}
\end{pmatrix}$ \\ 
\hline 
$(\zeta, \zeta^{9})$ & $\begin{pmatrix}
-1 & -1 & 0 & -1 & c_{0} \\
2 & 0 & -1 & 1 & c_{1} \\
-1 & 1 & 0 & -2 & c_{2} \\
1 & 0 & 1 & 1 & c_{3} \\
\end{pmatrix}$ & $\frac{1}{5} \begin{pmatrix}
2 c_{0} + 3 c_{1} + 2 c_{2} + 3 c_{3} \\
-4 c_{0} - c_{1} + c_{2} - c_{3} \\
c_{0} - c_{1} + c_{2} + 4 c_{3} \\
-3 c_{0} - 2 c_{1} - 3 c_{2} - 2 c_{3}
\end{pmatrix}$ \\ 
\hline 
$(\zeta^{3}, \zeta^{7})$ & $\begin{pmatrix}
0 & -1 & -2 & -1 & c_{0} \\
0 & 1 & 1 & -1 & c_{1} \\
1 & -1 & -1 & 0 & c_{2} \\
1 & 2 & 1 & 0 & c_{3}
\end{pmatrix}$ & $\frac{1}{5} \begin{pmatrix}
- c_{0} + c_{1} + 4 c_{2} + c_{3} \\
2 c_{0} - 2 c_{1} - 3 c_{2} + 3 c_{3} \\
-3 c_{0} + 3 c_{1} + 2 c_{2} - 2 c_{3} \\
- c_{0} - 4 c_{1} - c_{2} + c_{3}
\end{pmatrix}$ \\ 
\hline 
$(\zeta^{3}, \zeta^{9})$ & $\begin{pmatrix}
-1 & -2 & -1 & 0 & c_{0} \\
1 & 1 & -1 & -1 & c_{1} \\
-1 & -1 & 0 & -1 & c_{2} \\
2 & 1 & 0 & 0 & c_{3}
\end{pmatrix}$ & $\frac{1}{5} \begin{pmatrix}
c_{0} - c_{1} + c_{2} + 4 c_{3} \\
-2 c_{0} + 2 c_{1} - 2 c_{2} - 3 c_{3} \\
-2 c_{0} - 3 c_{1} + 3 c_{2} + 2 c_{3} \\
c_{0} - c_{1} - 4 c_{2} - c_{3}
\end{pmatrix}$ \\ 
\hline 
$(\zeta^{7}, \zeta^{9})$ & $\begin{pmatrix}
-1 & -1 & 1 & 1 & c_{0} \\
1 & 0 & -2 & 0 & c_{1} \\
-2 & 0 & 1 & -1 & c_){2} \\
1 & -1 & -1 & 0 & c_{3}
\end{pmatrix}$ & $\frac{1}{5} \begin{pmatrix}
-2 c_{0} - 3 c_{1} - 2 c_{2} + 2 c_{3} \\
- c_{0} + c_{1} - c_{2} - 4 c_{3} \\
- c_{0} - 4 c_{1} - c_{2} + c_{3} \\
3 c_{0} + 2 c_{1} - 2 c_{2} - 3 c_{3}
\end{pmatrix}$ \\
\hline
\end{longtable}

\small\addtolength{\tabcolsep}{0.8pt}
\begin{longtable}{|c|c|c|c|}
\caption{\textbf{For} $\bm{p=3}$ \textbf{and} $\bm{r=1,2,3};$  $\bm{n=2^{r}p.}$}
\label{table 2} \\
\hline
$(\sigma(\zeta), \tau(\zeta))$ & $(A ~ C)$ & $\text{det}(A)$ & $C'$ \\ 
\hline \hline
\multicolumn{4}{|c|}{$\bm{(r,p) = (1,3),}$ $\bm{n = 6,}$ $\bm{\phi(6) = |\{1, 5\}| = 2,}$ $\bm{\Phi_{6}(x) = x^{2} - x + 1.$}} \\
\hline \hline
$(\zeta, \zeta^{5})$ & $\begin{pmatrix}
-1 & -2 & c_{0} \\
2 & 1 & c_{1}
\end{pmatrix}$ & $3$ & $\frac{1}{3}\begin{pmatrix}
c_{0} + c_{1} \\
-c_{0} - c_{1}
\end{pmatrix}$ \\
\hline \hline
\multicolumn{4}{|c|}{$\bm{(r,p) = (2,3),}$ $\bm{n = 12,}$ $\bm{\phi(12) = |\{1, 5, 7, 11\}| = 4,}$ $\bm{\Phi_{12}(x) = x^{4} - x^{2} + 1.$}} \\
\hline \hline
$(\zeta, \zeta^{5})$ & $\begin{pmatrix}
0 & 1 & 0 & -1 & c_{0} \\
2 & 0 & 1 & 0 & c_{1} \\
0 & 1 & 0 & 2 & c_{2} \\
-1 & 0 & 1 & 0 & c_{3}
\end{pmatrix}$ & $9$ & $\frac{1}{3} \begin{pmatrix}
c_{1} - c_{3} \\
2 c_{0} + c_{2} \\
c_{1} + 2 c_{3} \\
- c_{0} + c_{2}
\end{pmatrix}$ \\ 
\hline 
$(\zeta, \zeta^{7})$ & $\begin{pmatrix}
0 & 0 & 0 & -2 & c_{0} \\
2 & 0 & 0 & 0 & c_{1} \\
0 & 2 & 0 & 2 & c_{2} \\
0 & 0 & 2 & 0 & c_{3}
\end{pmatrix}$ & $16$ & $\frac{1}{2} \begin{pmatrix}
c_{1} \\
c_{0} + c_{2} \\
c_{3} \\
- c_{0}
\end{pmatrix}$ \\ 
\hline 
$(\zeta, \zeta^{11})$ & $\begin{pmatrix}
0 & -1 & 0 & -1 & c_{0} \\
0 & 0 & -1 & 0 & c_{1} \\
0 & 1 & 0 & 0 & c_{2} \\
1 & 0 & 1 & 0 & c_{3}
\end{pmatrix}$ & $1$ & $\begin{pmatrix}
c_{1} + c_{3} \\
c_{2} \\
-c_{1} \\
-c_{0} - c_{2}
\end{pmatrix}$ \\ 
\hline 
$(\zeta^{5}, \zeta^{7})$ & $\begin{pmatrix}
0 & -1 & 0 & -1 & c_{0} \\
0 & 0 & -1 & 0 & c_{1} \\
0 & 1 & 0 & 0 & c_{2} \\
1 & 0 & 1 & 0 & c_{3}
\end{pmatrix}$ & $1$ & $\begin{pmatrix}
c_{1} + c_{3} \\
c_{2} \\
-c_{1} \\
-c_{0} - c_{2}
\end{pmatrix}$ \\ 
\hline 
$(\zeta^{5}, \zeta^{11})$ & $\begin{pmatrix}
0 & -2 & 0 & 0 & c_{0} \\
-2 & 0 & -2 & 0 & c_{1} \\
0 & 0 & 0 & -2 & c_{2} \\
2 & 0 & 0 & 0 & c_{3}
\end{pmatrix}$ & $16$ & $\frac{1}{2} \begin{pmatrix}
c_{3} \\
- c_{0} \\
- c_{1} - c_{3} \\
- c_{2}
\end{pmatrix}$ \\ 
\hline 
$(\zeta^{7}, \zeta^{11})$ & $\begin{pmatrix}
0 & -1 & 0 & 1 & c_{0} \\
-2 & 0 & -1 & 0 & c_{1} \\
0 & -1 & 0 & -2 & c_{2} \\
1 & 0 & -1 & 0 & c_{3}
\end{pmatrix}$ & $9$ & $\frac{1}{3} \begin{pmatrix}
- c_{1} + c_{3} \\
-2 c_{0} -  c_{2} \\
- c_{1} - 2 c_{3} \\
c_{0} - c_{2}
\end{pmatrix}$ \\ 
\hline \hline  
\multicolumn{4}{|c|}{$\bm{(r,p) = (3,3),}$ $\bm{n=24,}$ $\bm{\phi(24) = |\{1, 5, 7, 11, 13, 17, 19, 23\}| = 8,}$ $\bm{\Phi_{24}(x) = x^{8} - x^{4} + 1.$}} \\
\hline \hline 
$(\zeta, \zeta^{5})$ & $\begin{pmatrix}
0 & 0 & 0 & 1 & 0 & 0 & 0 & 0 & c_{0} \\
1 & 0 & 0 & 0 & 1 & 0 & 0 & 0 & c_{1} \\
0 & 1 & 0 & 0 & 0 & 1 & 0 & 0 & c_{2} \\
0 & 0 & 1 & 0 & 0 & 0 & 1 & 0 & c_{3} \\
0 & 0 & 0 & 0 & 0 & 0 & 0 & 1 & c_{4} \\
-1 & 0 & 0 & 0 & 0 & 0 & 0 & 0 & c_{5} \\
0 & -1 & 0 & 0 & 0 & 0 & 0 & 0 & c_{6} \\
0 & 0 & -1 & 0 & 0 & 0 & 0 & 0 & c_{7}
\end{pmatrix}$ & $1$ & $\begin{pmatrix}
-c_{5} \\
-c_{6} \\
-c_{7} \\
c_{0} \\
c_{1} + c_{5} \\
c_{2} + c_{6} \\
c_{3} + c_{7} \\
c_{4}
\end{pmatrix}$ \\ 
\hline 
$(\zeta, \zeta^{7})$ & $\begin{pmatrix}
0 & 1 & 0 & 0 & 0 & 1 & 0 & -1 & c_{0} \\
1 & 0 & 1 & 0 & 0 & 0 & 1 & 0 & c_{1} \\
0 & 1 & 0 & 1 & 0 & 0 & 0 & 1 & c_{2} \\
0 & 0 & 1 & 0 & 1 & 0 & 0 & 0 & c_{3} \\
0 & -1 & 0 & 1 & 0 & 0 & 0 & 1 & c_{4} \\
0 & 0 & -1 & 0 & 1 & 0 & 0 & 0 & c_{5} \\
0 & 0 & 0 & -1 & 0 & 1 & 0 & 0 & c_{6} \\
-1 & 0 & 0 & 0 & -1 & 0 & 1 & 0 & c_{7}
\end{pmatrix}$ & $16$ & $\frac{1}{2} \begin{pmatrix}
c_{1} - c_{3} - c_{7} \\
c_{2} - c_{4} \\
c_{3} - c_{5} \\
c_{0} + c_{4} - c_{6} \\
c_{3} + c_{5} \\
c_{0} + c_{4} + c_{6} \\
c_{1} + c_{5} + c_{7} \\
-c_{0} + c_{2} + c_{6}
\end{pmatrix}$ \\ 
\hline 
$(\zeta, \zeta^{11})$ & $\begin{pmatrix}
0 & 1 & 0 & 0 & 0 & 0 & 0 & -1 & c_{0} \\
1 & 0 & 1 & 0 & 0 & 0 & 0 & 0 & c_{1} \\
0 & 1 & 0 & 1 & 0 & 0 & 0 & 0 & c_{2} \\
1 & 0 & 1 & 0 & 1 & 0 & 0 & 0 & c_{3} \\
0 & 0 & 0 & 1 & 0 & 1 & 0 & 1 & c_{4} \\
0 & 0 & 0 & 0 & 1 & 0 & 1 & 0 & c_{5} \\
0 & 0 & 0 & 0 & 0 & 1 & 0 & 1 & c_{6} \\
-1 & 0 & 0 & 0 & 0 & 0 & 1 & 0 & c_{7}
\end{pmatrix}$ & $1$ & $\begin{pmatrix}
c_{1} - c_{3} + c_{5} - c_{7} \\
c_{2} - c_{4} + c_{6} \\
c_{3} - c_{5} + c_{7} \\
c_{4} - c_{6} \\
-c_{1} + c_{3} \\
c_{0} - c_{2} + c_{4} \\
c_{1} - c_{3} + c_{5} \\
-c_{0} + c_{2} - c_{4} + c_{6}
\end{pmatrix}$ \\ 
\hline 
$(\zeta, \zeta^{13})$ & $\begin{pmatrix}
0 & 0 & 0 & 0 & 0 & 0 & 0 & -2 & c_{0} \\
2 & 0 & 0 & 0 & 0 & 0 & 0 & 0 & c_{1} \\
0 & 2 & 0 & 0 & 0 & 0 & 0 & 0 & c_{2} \\
0 & 0 & 2 & 0 & 0 & 0 & 0 & 0 & c_{3} \\
0 & 0 & 0 & 2 & 0 & 0 & 0 & 2 & c_{4} \\
0 & 0 & 0 & 0 & 2 & 0 & 0 & 0 & c_{5} \\
0 & 0 & 0 & 0 & 0 & 2 & 0 & 0 & c_{6} \\
0 & 0 & 0 & 0 & 0 & 0 & 2 & 0 & c_{7}
\end{pmatrix}$ & $256$ & $\frac{1}{2} \begin{pmatrix}
c_{1} \\
c_{2} \\
c_{3} \\
-c_{0} + c_{4} \\
c_{5} \\
c_{6} \\
c_{7} \\
c_{0}
\end{pmatrix}$ \\ 
\hline 
$(\zeta, \zeta^{17})$ & $\begin{pmatrix}
0 & 0 & 0 & -1 & 0 & 0 & 0 & -2 & c_{0} \\
1 & 0 & 0 & 0 & -1 & 0 & 0 & 0 & c_{1} \\
0 & 1 & 0 & 0 & 0 & -1 & 0 & 0 & c_{2} \\
0 & 0 & 1 & 0 & 0 & 0 & -1 & 0 & c_{3} \\
0 & 0 & 0 & 2 & 0 & 0 & 0 & 1 & c_{4} \\
1 & 0 & 0 & 0 & 2 & 0 & 0 & 0 & c_{5} \\
0 & 1 & 0 & 0 & 0 & 2 & 0 & 0 & c_{6} \\
0 & 0 & 1 & 0 & 0 & 0 & 2 & 0 & c_{7}
\end{pmatrix}$ & $81$ & $\frac{1}{3} \begin{pmatrix}
2 c_{1} + c_{5} \\
2 c_{2} + c_{6} \\
2 c_{3} + c_{7} \\
c_{0} + 2 c_{4} \\
-c_{1} + c_{5} \\
-c_{2} + c_{6} \\
-c_{3} + c_{7} \\
-2 c_{0} - c_{4}
\end{pmatrix}$ \\ 
\hline 
$(\zeta, \zeta^{19})$ & $\begin{pmatrix}
0 & -1 & 0 & 0 & 0 & -1 & 0 & -1 & c_{0} \\
1 & 0 & -1 & 0 & 0 & 0 & -1 & 0 & c_{1} \\
0 & 1 & 0 & -1 & 0 & 0 & 0 & -1 & c_{2} \\
0 & 0 & 1 & 0 & -1 & 0 & 0 & 0 & c_{3} \\
0 & 1 & 0 & 1 & 0 & 0 & 0 & 1 & c_{4} \\
0 & 0 & 1 & 0 & 1 & 0 & 0 & 0 & c_{5} \\
0 & 0 & 0 & 1 & 0 & 1 & 0 & 0 & c_{6} \\
1 & 0 & 0 & 0 & 1 & 0 & 1 & 0 & c_{7}
\end{pmatrix}$ & $16$ & $\frac{1}{2} \begin{pmatrix}
c_{1} + c_{3} + c_{7} \\
c_{2} + c_{4} \\
c_{3} + c_{5} \\
c_{0} + c_{4} + c_{6} \\
-c_{3} + c_{5} \\
-c_{0} - c_{4} + c_{6} \\
-c_{1} - c_{5} + c_{7} \\
-c_{0} - c_{2} - c_{6}
\end{pmatrix}$ \\ 
\hline 
$(\zeta, \zeta^{23})$ & $\begin{pmatrix}
0 & -1 & 0 & 0 & 0 & 0 & 0 & -1 & c_{0} \\
1 & 0 & -1 & 0 & 0 & 0 & 0 & 0 & c_{1} \\
0 & 1 & 0 & -1 & 0 & 0 & 0 & 0 & c_{2} \\
-1 & 0 & 1 & 0 & -1 & 0 & 0 & 0 & c_{3} \\
0 & 0 & 0 & 1 & 0 & -1 & 0 & 1 & c_{4} \\
0 & 0 & 0 & 0 & 1 & 0 & -1 & 0 & c_{5} \\
0 & 0 & 0 & 0 & 0 & 1 & 0 & -1 & c_{6} \\
1 & 0 & 0 & 0 & 0 & 0 & 1 & 0 & c_{7}
\end{pmatrix}$ & $1$ & $\begin{pmatrix}
c_{1} - c_{7} \\
c_{0} + c_{2} \\
c_{1} + c_{3} \\
c_{0} + c_{2} + c_{4} \\
c_{3} + c_{5} + c_{7} \\
c_{4} + c_{6} \\
c_{5} + c_{7} \\
-c_{0} + c_{6}
\end{pmatrix}$ \\ 
\hline 
$(\zeta^{5}, \zeta^{7})$ & $\begin{pmatrix}
0 & 1 & 0 & -1 & 0 & 1 & 0 & -1 & c_{0} \\
0 & 0 & 1 & 0 & -1 & 0 & 1 & 0 & c_{1} \\
0 & 0 & 0 & 1 & 0 & -1 & 0 & 1 & c_{2} \\
0 & 0 & 0 & 0 & 1 & 0 & -1 & 0 & c_{3} \\
0 & -1 & 0 & 1 & 0 & 0 & 0 & 0 & c_{4} \\
1 & 0 & -1 & 0 & 1 & 0 & 0 & 0 & c_{5} \\
0 & 1 & 0 & -1 & 0 & 1 & 0 & 0 & c_{6} \\
-1 & 0 & 1 & 0 & -1 & 0 & 1 & 0 & c_{7}
\end{pmatrix}$ & $1$ & $\begin{pmatrix}
c_{1} - c_{7} \\
c_{0} + c_{2} \\
c_{1} + c_{3} \\
c_{0} + c_{2} + c_{4} \\
c_{3} + c_{5} + c_{7} \\
c_{4} + c_{6} \\
c_{5} + c_{7} \\
-c_{0} + c_{6}
\end{pmatrix}$ \\ 
\hline 
$(\zeta^{5}, \zeta^{11})$ & $\begin{pmatrix}
0 & 1 & 0 & -1 & 0 & 0 & 0 & -1 & c_{0} \\
0 & 0 & 1 & 0 & -1 & 0 & 0 & 0 & c_{1} \\
0 & 0 & 0 & 1 & 0 & -1 & 0 & 0 & c_{2} \\
1 & 0 & 0 & 0 & 1 & 0 & -1 & 0 & c_{3} \\
0 & 0 & 0 & 1 & 0 & 1 & 0 & 0 & c_{4} \\
1 & 0 & 0 & 0 & 1 & 0 & 1 & 0 & c_{5} \\
0 & 1 & 0 & 0 & 0 & 1 & 0 & 1 & c_{6} \\
-1 & 0 & 1 & 0 & 0 & 0 & 1 & 0 & c_{7}
\end{pmatrix}$ & $16$ & $\frac{1}{2} \begin{pmatrix}
c_{1} + c_{5} - c_{7} \\
c_{0} + c_{2} + c_{6} \\
c_{1} + c_{3} + c_{7} \\
c_{2} + c_{4} \\
-c_{1} + c_{3} + c_{7} \\
-c_{2} + c_{4} \\
-c_{3} + c_{5} \\
-c_{0} - c_{4} + c_{6}
\end{pmatrix}$ \\ 
\hline 
$(\zeta^{5}, \zeta^{13})$ & $\begin{pmatrix}
0 & 0 & 0 & -1 & 0 & 0 & 0 & -2 & c_{0} \\
1 & 0 & 0 & 0 & -1 & 0 & 0 & 0 & c_{1} \\
0 & 1 & 0 & 0 & 0 & -1 & 0 & 0 & c_{2} \\
0 & 0 & 1 & 0 & 0 & 0 & -1 & 0 & c_{3} \\
0 & 0 & 0 & 2 & 0 & 0 & 0 & 1 & c_{4} \\
1 & 0 & 0 & 0 & 2 & 0 & 0 & 0 & c_{5} \\
0 & 1 & 0 & 0 & 0 & 2 & 0 & 0 & c_{6} \\
0 & 0 & 1 & 0 & 0 & 0 & 2 & 0 & c_{7}
\end{pmatrix}$ & $81$ & $\frac{1}{3} \begin{pmatrix}
2 c_{1} + c_{5} \\
2 c_{2} + c_{6} \\
2 c_{3} + c_{7} \\
c_{0} + 2 c_{4} \\
-c_{1} + c_{5} \\
-c_{2} + c_{6} \\
-c_{3} + c_{7} \\
-2c_{0} - c_{4}
\end{pmatrix}$ \\ 
\hline 
$(\zeta^{5}, \zeta^{17})$ & $\begin{pmatrix}
0 & 0 & 0 & -2 & 0 & 0 & 0 & -2 & c_{0} \\
0 & 0 & 0 & 0 & -2 & 0 & 0 & 0 & c_{1} \\
0 & 0 & 0 & 0 & 0 & -2 & 0 & 0 & c_{2} \\
0 & 0 & 0 & 0 & 0 & 0 & -2 & 0 & c_{3} \\
0 & 0 & 0 & 2 & 0 & 0 & 0 & 0 & c_{4} \\
2 & 0 & 0 & 0 & 2 & 0 & 0 & 0 & c_{5} \\
0 & 2 & 0 & 0 & 0 & 2 & 0 & 0 & c_{6} \\
0 & 0 & 2 & 0 & 0 & 0 & 2 & 0 & c_{7}
\end{pmatrix}$ & $256$ & $\frac{1}{2} \begin{pmatrix}
c_{1} + c_{5} \\
c_{2} + c_{6} \\
c_{3} + c_{7} \\
c_{4} \\
-c_{1} \\
-c_{2} \\
-c_{3} \\
-c_{0} - c_{4}
\end{pmatrix}$ \\ 
\hline 
$(\zeta^{5}, \zeta^{19})$ & $\begin{pmatrix}
0 & -1 & 0 & -1 & 0 & -1 & 0 & -1 & c_{0} \\
0 & 0 & -1 & 0 & -1 & 0 & -1 & 0 & c_{1} \\
0 & 0 & 0 & -1 & 0 & -1 & 0 & -1 & c_{2} \\
0 & 0 & 0 & 0 & -1 & 0 & -1 & 0 & c_{3} \\
0 & 1 & 0 & 1 & 0 & 0 & 0 & 0 & c_{4} \\
1 & 0 & 1 & 0 & 1 & 0 & 0 & 0 & c_{5} \\
0 & 1 & 0 & 1 & 0 & 1 & 0 & 0 & c_{6} \\
1 & 0 & 1 & 0 & 1 & 0 & 1 & 0 & c_{7} \\
\end{pmatrix}$ & $1$ & $\begin{pmatrix}
c_{1} + c_{7} \\
-c_{0} + c_{2} \\
-c_{1} + c_{3} \\
c_{0} - c_{2} + c_{4} \\
-c_{3} + c_{5} - c_{7} \\
-c_{4} + c_{6} \\
-c_{5} + c_{7} \\
-c_{0} - c_{6}
\end{pmatrix}$ \\ 
\hline 
$(\zeta^{5}, \zeta^{23})$ & $\begin{pmatrix}
0 & -1 & 0 & -1 & 0 & 0 & 0 & -1 & c_{0} \\
0 & 0 & -1 & 0 & -1 & 0 & 0 & 0 & c_{1} \\
0 & 0 & 0 & -1 & 0 & -1 & 0 & 0 & c_{2} \\
-1 & 0 & 0 & 0 & -1 & 0 & -1 & 0 & c_{3} \\
0 & 0 & 0 & 1 & 0 & -1 & 0 & 0 & c_{4} \\
1 & 0 & 0 & 0 & 1 & 0 & -1 & 0 & c_{5} \\
0 & 1 & 0 & 0 & 0 & 1 & 0 & -1 & c_{6} \\
1 & 0 & 1 & 0 & 0 & 0 & 1 & 0 & c_{7}
\end{pmatrix}$ & $16$ & $\frac{1}{2} \begin{pmatrix}
c_{1} + c_{5} + c_{7} \\
-c_{0} + c_{2} + c_{6} \\
-c_{1} + c_{3} + c_{7} \\
-c_{2} + c_{4} \\
-c_{1} - c_{3} - c_{7} \\
-c_{2} - c_{4} \\
-c_{3} - c_{5} \\
-c_{0} - c_{4} - c_{6}
\end{pmatrix}$ \\ 
\hline 
$(\zeta^{7}, \zeta^{11})$ & $\begin{pmatrix}
0 & 0 & 0 & 0 & 0 & -1 & 0 & 0 & c_{0} \\
0 & 0 & 0 & 0 & 0 & 0 & -1 & 0 & c_{1} \\
0 & 0 & 0 & 0 & 0 & 0 & 0 & -1 & c_{2} \\
1 & 0 & 0 & 0 & 0 & 0 & 0 & 0 & c_{3} \\
0 & 1 & 0 & 0 & 0 & 1 & 0 & 0 & c_{4} \\
0 & 0 & 1 & 0 & 0 & 0 & 1 & 0 & c_{5} \\
0 & 0 & 0 & 1 & 0 & 0 & 0 & 1 & c_{6} \\
0 & 0 & 0 & 0 & 1 & 0 & 0 & 0 & c_{7}
\end{pmatrix}$ & $1$ & $\begin{pmatrix}
c_{3} \\
c_{0} + c_{4} \\
c_{1} + c_{5} \\
c_{2} + c_{6} \\
c_{7} \\
-c_{0} \\
-c_{1} \\
-c_{2}
\end{pmatrix}$ \\ 
\hline 
$(\zeta^{7}, \zeta^{13})$ & $\begin{pmatrix}
0 & -1 & 0 & 0 & 0 & -1 & 0 & -1 & c_{0} \\
1 & 0 & -1 & 0 & 0 & 0 & -1 & 0 & c_{1} \\
0 & 1 & 0 & -1 & 0 & 0 & 0 & -1 & c_{2} \\
0 & 0 & 1 & 0 & -1 & 0 & 0 & 0 & c_{3} \\
0 & 1 & 0 & 1 & 0 & 0 & 0 & 1 & c_{4} \\
0 & 0 & 1 & 0 & 1 & 0 & 0 & 0 & c_{5} \\
0 & 0 & 0 & 1 & 0 & 1 & 0 & 0 & c_{6} \\
1 & 0 & 0 & 0 & 1 & 0 & 1 & 0 & c_{7}
\end{pmatrix}$ & $16$ & $\frac{1}{2} \begin{pmatrix}
c_{1} + c_{3} + c_{7} \\
c_{2} + c_{4} \\
c_{3} + c_{5} \\
c_{0} + c_{4} + c_{6} \\
-c_{3} + c_{5} \\
-c_{0} - c_{4} + c_{6} \\
-c_{1} - c_{5} + c_{7} \\
-c_{0} - c_{2} - c_{6}
\end{pmatrix}$ \\ 
\hline 
$(\zeta^{7}, \zeta^{17})$ & $\begin{pmatrix}
0 & -1 & 0 & -1 & 0 & -1 & 0 & -1 & c_{0} \\
0 & 0 & -1 & 0 & -1 & 0 & -1 & 0 & c_{1} \\
0 & 0 & 0 & -1 & 0 & -1 & 0 & -1 & c_{2} \\
0 & 0 & 0 & 0 & -1 & 0 & -1 & 0 & c_{3} \\
0 & 1 & 0 & 1 & 0 & 0 & 0 & 0 & c_{4} \\
1 & 0 & 1 & 0 & 1 & 0 & 0 & 0 & c_{5} \\
0 & 1 & 0 & 1 & 0 & 1 & 0 & 0 & c_{6} \\
1 & 0 & 1 & 0 & 1 & 0 & 1 & 0 & c_{7}
\end{pmatrix}$ & $1$ & $\begin{pmatrix}
c_{1} + c_{7} \\
-c_{0} + c_{2} \\
-c_{1} + c_{3} \\
c_{0} - c_{2} + c_{4} \\
-c_{3} + c_{5} - c_{7} \\
-c_{4} + c_{6} \\
-c_{5} + c_{7} \\
-c_{0} - c_{6}
\end{pmatrix}$ \\ 
\hline 
$(\zeta^{7}, \zeta^{19})$ & $\begin{pmatrix}
0 & -2 & 0 & 0 & 0 & -2 & 0 & 0 & c_{0} \\
0 & 0 & -2 & 0 & 0 & 0 & -2 & 0 & c_{1} \\
0 & 0 & 0 & -2 & 0 & 0 & 0 & -2 & c_{2} \\
0 & 0 & 0 & 0 & -2 & 0 & 0 & 0 & c_{3} \\
0 & 2 & 0 & 0 & 0 & 0 & 0 & 0 & c_{4} \\
0 & 0 & 2 & 0 & 0 & 0 & 0 & 0 & c_{5} \\
0 & 0 & 0 & 2 & 0 & 0 & 0 & 0 & c_{6} \\
2 & 0 & 0 & 0 & 2 & 0 & 0 & 0 & c_{7}
\end{pmatrix}$ & $256$ & $\frac{1}{2} \begin{pmatrix}
c_{3} + c_{7} \\
c_{4} \\
c_{5} \\
c_{6} \\
-c_{3} \\
-c_{0} - c_{4} \\
-c_{1} - c_{5} \\
-c_{2} - c_{6}
\end{pmatrix}$ \\ 
\hline 
$(\zeta^{7}, \zeta^{23})$ & $\begin{pmatrix}
0 & -2 & 0 & 0 & 0 & -1 & 0 & 0 & c_{0} \\
0 & 0 & -2 & 0 & 0 & 0 & -1 & 0 & c_{1} \\
0 & 0 & 0 & -2 & 0 & 0 & 0 & -1 & c_{2} \\
-1 & 0 & 0 & 0 & -2 & 0 & 0 & 0 & c_{3} \\
0 & 1 & 0 & 0 & 0 & -1 & 0 & 0 & c_{4} \\
0 & 0 & 1 & 0 & 0 & 0 & -1 & 0 & c_{5} \\
0 & 0 & 0 & 1 & 0 & 0 & 0 & -1 & c_{6} \\
2 & 0 & 0 & 0 & 1 & 0 & 0 & 0 & c_{7}
\end{pmatrix}$ & $81$ & $\frac{1}{3} \begin{pmatrix}
c_{3} + 2 c_{7} \\
-c_{0} + c_{4} \\
-c_{1} + c_{5} \\
-c_{2} + c_{6} \\
-2 c_{3} - c_{7} \\
-c_{0} - 2 c_{4} \\
-c_{1} - 2 c_{5} \\
-c_{2} - 2 c_{6}
\end{pmatrix}$ \\ 
\hline 
$(\zeta^{11}, \zeta^{13})$ & $\begin{pmatrix}
0 & -1 & 0 & 0 & 0 & 0 & 0 & -1 & c_{0} \\
1 & 0 & -1 & 0 & 0 & 0 & 0 & 0 & c_{1} \\
0 & 1 & 0 & -1 & 0 & 0 & 0 & 0 & c_{2} \\
-1 & 0 & 1 & 0 & -1 & 0 & 0 & 0 & c_{3} \\
0 & 0 & 0 & 1 & 0 & -1 & 0 & 1 & c_{4} \\
0 & 0 & 0 & 0 & 1 & 0 & -1 & 0 & c_{5} \\
0 & 0 & 0 & 0 & 0 & 1 & 0 & -1 & c_{6} \\
1 & 0 & 0 & 0 & 0 & 0 & 1 & 0 & c_{7}
\end{pmatrix}$ & $1$ & $\begin{pmatrix}
c_{1} + c_{3} + c_{5} + c_{7} \\
c_{2} + c_{4} + c_{6} \\
c_{3} + c_{5} + c_{7} \\
c_{4} + c_{6} \\
-c_{1} - c_{3} \\
-c_{0} - c_{2} - c_{4} \\
-c_{1} - c_{3} - c_{5} \\
-c_{0} - c_{2} - c_{4} - c_{6}
\end{pmatrix}$ \\ 
\hline 
$(\zeta^{11}, \zeta^{17})$ & $\begin{pmatrix}
0 & -1 & 0 & -1 & 0 & 0 & 0 & -1 & c_{0} \\
0 & 0 & -1 & 0 & -1 & 0 & 0 & 0 & c_{1} \\
0 & 0 & 0 & -1 & 0 & -1 & 0 & 0 & c_{2} \\
-1 & 0 & 0 & 0 & -1 & 0 & -1 & 0 & c_{3} \\
0 & 0 & 0 & 1 & 0 & -1 & 0 & 0 & c_{4} \\
1 & 0 & 0 & 0 & 1 & 0 & -1 & 0 & c_{5} \\
0 & 1 & 0 & 0 & 0 & 1 & 0 & -1 & c_{6} \\
1 & 0 & 1 & 0 & 0 & 0 & 1 & 0 & c_{7}
\end{pmatrix}$ & $16$ & $\frac{1}{2} \begin{pmatrix}
c_{1} + c_{5} + c_{7} \\
-c_{0} + c_{2} + c_{6} \\
-c_{1} + c_{3} + c_{7} \\
-c_{2} + c_{4} \\
-c_{1} - c_{3} - c_{7} \\
-c_{2} - c_{4} \\
-c_{3} - c_{5} \\
-c_{0} - c_{4} - c_{6}
\end{pmatrix}$ \\ 
\hline 
$(\zeta^{11}, \zeta^{19})$ & $\begin{pmatrix}
0 & -2 & 0 & 0 & 0 & -1 & 0 & 0 & c_{0} \\
0 & 0 & -2 & 0 & 0 & 0 & -1 & 0 & c_{1} \\
0 & 0 & 0 & -2 & 0 & 0 & 0 & -1 & c_{2} \\
-1 & 0 & 0 & 0 & -2 & 0 & 0 & 0 & c_{3} \\
0 & 1 & 0 & 0 & 0 & -1 & 0 & 0 & c_{4} \\
0 & 0 & 1 & 0 & 0 & 0 & -1 & 0 & c_{5} \\
0 & 0 & 0 & 1 & 0 & 0 & 0 & -1 & c_{6} \\
2 & 0 & 0 & 0 & 1 & 0 & 0 & 0 & c_{7}
\end{pmatrix}$ & $81$ & $\frac{1}{3} \begin{pmatrix}
c_{3} + 2 c_{7} \\
-c_{0} + c_{4} \\
-c_{1} + c_{5} \\
-c_{2} + c_{6} \\
-2 c_{3} - c_{7} \\
-c_{0} - 2 c_{4} \\
-c_{1} - 2 c_{5} \\
-c_{2} - 2 c_{6}
\end{pmatrix}$ \\ 
\hline 
$(\zeta^{11}, \zeta^{23})$ & $\begin{pmatrix}
0 & -2 & 0 & 0 & 0 & 0 & 0 & 0 & c_{0} \\
0 & 0 & -2 & 0 & 0 & 0 & 0 & 0 & c_{1} \\
0 & 0 & 0 & -2 & 0 & 0 & 0 & 0 & c_{2} \\
-2 & 0 & 0 & 0 & -2 & 0 & 0 & 0 & c_{3} \\
0 & 0 & 0 & 0 & 0 & -2 & 0 & 0 & c_{4} \\
0 & 0 & 0 & 0 & 0 & 0 & -2 & 0 & c_{5} \\
0 & 0 & 0 & 0 & 0 & 0 & 0 & -2 & c_{6} \\
2 & 0 & 0 & 0 & 0 & 0 & 0 & 0 & c_{7}
\end{pmatrix}$ & $256$ & $\frac{1}{2} \begin{pmatrix}
c_{7} \\
-c_{0} \\
-c_{1} \\
-c_{2} \\
-c_{3} - c_{7} \\
-c_{4} \\
-c_{5} \\
-c_{6}
\end{pmatrix}$ \\ 
\hline 
$(\zeta^{13}, \zeta^{17})$ & $\begin{pmatrix}
0 & 0 & 0 & -1 & 0 & 0 & 0 & 0 & c_{0} \\
-1 & 0 & 0 & 0 & -1 & 0 & 0 & 0 & c_{1} \\
0 & -1 & 0 & 0 & 0 & -1 & 0 & 0 & c_{2} \\
0 & 0 & -1 & 0 & 0 & 0 & -1 & 0 & c_{3} \\
0 & 0 & 0 & 0 & 0 & 0 & 0 & -1 & c_{4} \\
1 & 0 & 0 & 0 & 0 & 0 & 0 & 0 & c_{5} \\
0 & 1 & 0 & 0 & 0 & 0 & 0 & 0 & c_{6} \\
0 & 0 & 1 & 0 & 0 & 0 & 0 & 0 & c_{7}
\end{pmatrix}$ & $1$ & $\begin{pmatrix}
c_{5} \\
c_{6} \\
c_{7} \\
-c_{0} \\
-c_{1} - c_{5} \\
-c_{2} - c_{6} \\
-c_{3} - c_{7} \\
-c_{4}
\end{pmatrix}$ \\ 
\hline 
$(\zeta^{13}, \zeta^{19})$ & $\begin{pmatrix}
0 & -1 & 0 & 0 & 0 & -1 & 0 & 1 & c_{0} \\
-1 & 0 & -1 & 0 & 0 & 0 & -1 & 0 & c_{1} \\
0 & -1 & 0 & -1 & 0 & 0 & 0 & -1 & c_{2} \\
0 & 0 & -1 & 0 & -1 & 0 & 0 & 0 & c_{3} \\
0 & 1 & 0 & -1 & 0 & 0 & 0 & -1 & c_{4} \\
0 & 0 & 1 & 0 & -1 & 0 & 0 & 0 & c_{5} \\
0 & 0 & 0 & 1 & 0 & -1 & 0 & 0 & c_{6} \\
1 & 0 & 0 & 0 & 1 & 0 & -1 & 0 & c_{7}
\end{pmatrix}$ & $16$ & $\frac{1}{2} \begin{pmatrix}
-c_{1} + c_{3} + c_{7} \\
-c_{2} + c_{4} \\
-c_{3} + c_{5} \\
-c_{0} - c_{4} + c_{6} \\
-c_{3} - c_{5} \\
-c_{0} - c_{4} - c_{6} \\
-c_{1} - c_{5} - c_{7} \\
c_{0} - c_{2} - c_{6}
\end{pmatrix}$ \\ 
\hline 
$(\zeta^{13}, \zeta^{23})$ & $\begin{pmatrix}
0 & -1 & 0 & 0 & 0 & 0 & 0 & 1 & c_{0} \\
-1 & 0 & -1 & 0 & 0 & 0 & 0 & 0 & c_{1} \\
0 & -1 & 0 & -1 & 0 & 0 & 0 & 0 & c_{2} \\
-1 & 0 & -1 & 0 & -1 & 0 & 0 & 0 & c_{3} \\
0 & 0 & 0 & -1 & 0 & -1 & 0 & -1 & c_{4} \\
0 & 0 & 0 & 0 & -1 & 0 & -1 & 0 & c_{5} \\
0 & 0 & 0 & 0 & 0 & -1 & 0 & -1 & c_{6} \\
1 & 0 & 0 & 0 & 0 & 0 & -1 & 0 & c_{7}
\end{pmatrix}$ & $1$ & $\begin{pmatrix}
-c_{1} + c_{3} - c_{5} + c_{7} \\
-c_{2} + c_{4} - c_{6} \\
-c_{3} + c_{5} - c_{7} \\
-c_{4} + c_{6} \\
c_{1} - c_{3} \\
-c_{0} + c_{2} - c_{4} \\
-c_{1} + c_{3} - c_{5} \\
c_{0} - c_{2} + c_{4} - c_{6}
\end{pmatrix}$ \\ 
\hline 
$(\zeta^{17}, \zeta^{19})$ & $\begin{pmatrix}
0 & -1 & 0 & 1 & 0 & -1 & 0 & 1 & c_{0} \\
0 & 0 & -1 & 0 & 1 & 0 & -1 & 0 & c_{1} \\
0 & 0 & 0 & -1 & 0 & 1 & 0 & -1 & c_{2} \\
0 & 0 & 0 & 0 & -1 & 0 & 1 & 0 & c_{3} \\
0 & 1 & 0 & -1 & 0 & 0 & 0 & 0 & c_{4} \\
-1 & 0 & 1 & 0 & -1 & 0 & 0 & 0 & c_{5} \\
0 & -1 & 0 & 1 & 0 & -1 & 0 & 0 & c_{6} \\
1 & 0 & -1 & 0 & 1 & 0 & -1 & 0 & c_{7}
\end{pmatrix}$ & $1$ & $\begin{pmatrix}
-c_{1} + c_{7} \\
-c_{0} - c_{2} \\
-c_{1} - c_{3} \\
-c_{0} - c_{2} - c_{4} \\
-c_{3} - c_{5} - c_{7} \\
-c_{4} - c_{6} \\
-c_{5} - c_{7} \\
c_{0} - c_{6}
\end{pmatrix}$ \\ 
\hline 
$(\zeta^{17}, \zeta^{23})$ & $\begin{pmatrix}
0 & -1 & 0 & 1 & 0 & 0 & 0 & 1 & c_{0} \\
0 & 0 & -1 & 0 & 1 & 0 & 0 & 0 & c_{1} \\
0 & 0 & 0 & -1 & 0 & 1 & 0 & 0 & c_{2} \\
-1 & 0 & 0 & 0 & -1 & 0 & 1 & 0 & c_{3} \\
0 & 0 & 0 & -1 & 0 & -1 & 0 & 0 & c_{4} \\
-1 & 0 & 0 & 0 & -1 & 0 & -1 & 0 & c_{5} \\
0 & -1 & 0 & 0 & 0 & -1 & 0 & -1 & c_{6} \\
1 & 0 & -1 & 0 & 0 & 0 & -1 & 0 & c_{7}
\end{pmatrix}$ & $16$ & $\frac{1}{2} \begin{pmatrix}
-c_{1} - c_{5} + c_{7} \\
-c_{0} - c_{2} - c_{6} \\
-c_{1} - c_{3} - c_{7} \\
-c_{2} - c_{4} \\
c_{1} - c_{3} - c_{7} \\
c_{2} - c_{4} \\
c_{3} - c_{5} \\
c_{0} + c_{4} - c_{6}
\end{pmatrix}$ \\ 
\hline 
$(\zeta^{19}, \zeta^{23})$ & $\begin{pmatrix}
0 & 0 & 0 & 0 & 0 & 1 & 0 & 0 & c_{0} \\
0 & 0 & 0 & 0 & 0 & 0 & 1 & 0 & c_{1} \\
0 & 0 & 0 & 0 & 0 & 0 & 0 & 1 & c_{2} \\
-1 & 0 & 0 & 0 & 0 & 0 & 0 & 0 & c_{3} \\
0 & -1 & 0 & 0 & 0 & -1 & 0 & 0 & c_{4} \\
0 & 0 & -1 & 0 & 0 & 0 & -1 & 0 & c_{5} \\
0 & 0 & 0 & -1 & 0 & 0 & 0 & -1 & c_{6} \\
0 & 0 & 0 & 0 & -1 & 0 & 0 & 0 & c_{7}
\end{pmatrix}$ & $1$ & $\begin{pmatrix}
-c_{3} \\
-c_{0} - c_{4} \\
-c_{1} - c_{5} \\
-c_{2} - c_{6} \\
-c_{7} \\
c_{0} \\
c_{1} \\
c_{2}
\end{pmatrix}$ \\ 
\hline 
\end{longtable}

\subsection{For $n = p^{k}$, $k \in \mathbb{N} \setminus \{1\}$ and $p$ a rational prime}\label{subsection 4.2}
In this subsection, we conjecture a characterization for a $(\sigma, \tau)$-derivation $D:O_{\mathbb{K}} \rightarrow O_{\mathbb{K}}$ to be inner when $n$ is of the form $n = p^{k}$, where $k \in \mathbb{N} \setminus \{1\}$ and $p$ is any rational prime. Then the degree of $\mathbb{K}$ over $\mathbb{Q}$ is $\phi(n) = \phi(p^{k}) = p^{k-1}(p-1)$.

The following \th\ref{conjecture 4.3} gives the determinant of the coefficient matrix formed in a particular way.

\begin{conjecture}\th\label{conjecture 4.3}
Let $n=p^{k}$ ($k \in \mathbb{N} \setminus \{1\}$ and $p$ is a rational prime). Suppose $\beta = \sum_{i=0}^{p^{k-1}(p-1)-1} b_{i} \zeta^{i} \in O_{\mathbb{K}}$ and $\beta (\tau - \sigma) (\zeta) =  \sum_{i=0}^{p^{k-1}(p-1)-1} \left( \sum_{j=0}^{p^{k-1}(p-1)-1} a_{ij} b_{j} \right) \zeta^{i}$. Further, let $v-u = p^{e_{1}} m$, where $m$ is a positive integer, $e_{1}$ is a non-negative integer, and $p$ does not divide $m$. Then $A = [a_{ij}]$ is a $p^{k-1}(p-1) \times p^{k-1}(p-1)$ matrix with determinant $\pm p^{p^{e_{1}}}$.
\end{conjecture}

Again, the sign of the determinant of the matrix $A$ depends on the order in which the rows of $A$ are formed. As a consequence, we propose another \th\ref{conjecture 4.4} which gives a necessary and sufficient condition for a $(\sigma, \tau)$-derivation $D: O_{\mathbb{K}} \rightarrow O_{\mathbb{K}}$ to be inner.

\begin{conjecture}\th\label{conjecture 4.4}
Under the hypotheses of \th\ref{conjecture 4.3}, a $(\sigma, \tau)$-derivation $D:O_{\mathbb{K}} \rightarrow O_{\mathbb{K}}$ with $D(\zeta) = \sum_{i=0}^{p^{k-1}(p-1)-1} c_{i} \zeta^{i} \in O_{\mathbb{K}}$, is inner if and only if $\frac{1}{p^{e_{1}}}(Adj(A)C) \in \mathbb{Z}^{p^{k-1}(p-1)}$. In particular, if $p^{e_{1}}$ divides $c_{i}$ for each $i \in \{0, 1, ..., p^{k-1}(p-1)-1\}$, then $D$ is inner. In fact, if $p$ divides $c_{i}$ for each $i \in \{0, 1, ..., p^{k-1}(p-1)-1\}$, then $D$ is inner.
\end{conjecture}

The case where $k=1$ was dealt with in \cite{Manju2023b}.
In this article, \th\ref{conjecture 4.3} and \th\ref{conjecture 4.4} have been verified for the pairs $(k,p) = (2,2), (3,2), (4,2), (5,2), (2,3), (3,3), (2,5)$. As $k$ and $p$ increase, the size of the matrix $A$ ($\phi(n) \times \phi(n)$, where $n=p^{k}$) increases, and it becomes computationally tedious for us to deal with matrices of such enormous size. For an illustration of the conjectures, we record our observations for the prime $p=2$ for $k=2, 3, 4$ and $p=3$ for $k=2$ in the following tables \ref{table 3} and \ref{table 4}, respectively.

\small\addtolength{\tabcolsep}{-3pt}
\begin{longtable}{|c|c|c|c|}
\caption{\textbf{For} $\bm{p=2}$ \textbf{and} $\bm{k=2,3,4;}$ $\bm{n=p^{k}.}$}
\label{table 3} \\
\hline 
$(\sigma(\zeta), \tau(\zeta))$ & $(A ~ C)$ & $\textbf{det}(A)$ & $C'$  \\
\hline \hline
\multicolumn{4}{|c|}{$\bm{(k,p) = (2,2),}$ $\bm{n = 4,}$ $\bm{\phi(4) = |\{1,3\}| = 2,}$ $\bm{\Phi_{4}(x) = x^{2} + 1.}$} \\
\hline \hline
$(\zeta, \zeta^{3})$ & $\begin{pmatrix}
0 & -2 & c_{0} \\
2 & 0 & c_{1}
\end{pmatrix}$ & $4$ & $\frac{1}{2} \begin{pmatrix}
c_{1} \\
-c_{0}
\end{pmatrix}$ \\ 
\hline \hline
\multicolumn{4}{|c|}{$\bm{(k,p) = (3,2),}$ $\bm{n = 8,}$ $\bm{\phi(8) = |\{[1, 3, 5, 7\}| = 4,}$ $\bm{\Phi_{8}(x) = x^{4} + 1.}$} \\
\hline \hline
$(\zeta, \zeta^{3})$ & $\begin{pmatrix}
0 & 1 & 0 & -1 & c_{0} \\
1 & 0 & 1 & 0 & c_{1} \\
0 & 1 & 0 & 1 & c_{2} \\
-1 & 0 & 1 & 0 & c_{3}
\end{pmatrix}$ & $4$ & $\frac{1}{2} \begin{pmatrix}
c_{1} - c_{3} \\
c_{0} + c_{2} \\
c_{1} + c_{3} \\
-c_{0} + c_{2}
\end{pmatrix}$ \\ 
\hline 
$(\zeta, \zeta^{5})$ & $\begin{pmatrix}
0 & 0 & 0 & -2 & c_{0} \\
2 & 0 & 0 & 0 & c_{1} \\
0 & 2 & 0 & 0 & c_{2} \\
0 & 0 & 2 & 0 & c_{3}
\end{pmatrix}$ & $16$ & $\frac{1}{2} \begin{pmatrix}
c_{1} \\
c_{2} \\
c_{3} \\
-c_{0}
\end{pmatrix}$ \\ 
\hline 
$(\zeta, \zeta^{7})$ & $\begin{pmatrix}
0 & -1 & 0 & -1 & c_{0} \\
1 & 0 & -1 & 0 & c_{1} \\
0 & 1 & 0 & -1 & c_{2} \\
1 & 0 & 1 & 0 & c_{3}
\end{pmatrix}$ & $4$ & $\frac{1}{2} \begin{pmatrix}
c_{1} + c_{3} \\
-c_{0} + c_{2} \\
-c_{1} + c_{3} \\
-c_{0} - c_{2}
\end{pmatrix}$ \\ 
\hline 
$(\zeta^{3}, \zeta^{5})$ & $\begin{pmatrix}
0 & -1 & 0 & -1 & c_{0} \\
1 & 0 & -1 & 0 & c_{1} \\
0 & 1 & 0 & -1 & c_{2} \\
1 & 0 & 1 & 0 & c_{3}
\end{pmatrix}$ & $4$ & $\frac{1}{2} \begin{pmatrix}
c_{1} + c_{3} \\
-c_{0} + c_{2} \\
-c_{1} + c_{3} \\
-c_{0} - c_{2}
\end{pmatrix}$ \\ 
\hline 
$(\zeta^{3}, \zeta^{7})$ & $\begin{pmatrix}
0 & -2 & 0 & 0 & c_{0} \\
0 & 0 & -2 & 0 & c_{1} \\
0 & 0 & 0 & -2 & c_{2} \\
2 & 0 & 0 & 0 & c_{3}
\end{pmatrix}$ & $16$ & $\frac{1}{2} \begin{pmatrix}
c_{3} \\
-c_{0} \\
-c_{1} \\
-c_{2}
\end{pmatrix}$ \\ 
\hline 
$(\zeta^{5}, \zeta^{7})$ & $\begin{pmatrix}
0 & -1 & 0 & 1 & c_{0} \\
-1 & 0 & -1 & 0 & c_{1} \\
0 & -1 & 0 & -1 & c_{2} \\
1 & 0 & -1 & 0 & c_{3}
\end{pmatrix}$ & $4$ & $\frac{1}{2} \begin{pmatrix}
-c_{1} + c_{3} \\
-c_{0} - c_{2} \\
-c_{1} - c_{3} \\
c_{0} - c_{2}
\end{pmatrix}$ \\ 
\hline \hline
\multicolumn{4}{|c|}{$\bm{(k,p) = (4,2),}$ $\bm{n = 16,}$ $\bm{\phi(16) = |\{[1, 3, 5, 7, 9, 11, 13, 15\}| = 8,}$ $\bm{\Phi_{16}(x) = x^{8} + 1.}$} \\
\hline  \hline
$(\sigma(\zeta), \tau(\zeta))$ & $(A ~ C)$ & $\text{det}(A)$ & $C'$ \\
\hline
$(\zeta, \zeta^{3})$ & $\begin{pmatrix}
0 & 0 & 0 & 0 & -1 & 0 & 1 & 0 & c_{0} \\
0 & 0 & 0 & -1 & 0 & 1 & 0 & 0 & c_{1} \\
0 & 0 & -1 & 0 & 1 & 0 & 0 & 0 & c_{2} \\
0 & -1 & 0 & 1 & 0 & 0 & 0 & 0 & c_{3} \\
-1 & 0 & 1 & 0 & 0 & 0 & 0 & 0 & c_{4} \\
0 & 1 & 0 & 0 & 0 & 0 & 0 & 1 & c_{5} \\
1 & 0 & 0 & 0 & 0 & 0 & 1 & 0 & c_{6} \\
0 & 0 & 0 & 0 & 0 & 1 & 0 & -1 & c_{7}
\end{pmatrix}$ & $4$ & $\frac{1}{2} \begin{pmatrix}
-c_{0} - c_{2} - c_{4} + c_{6} \\
-c_{1} - c_{3} + c_{5} + c_{7} \\
-c_{0} - c_{2} + c_{4} + c_{6} \\
-c_{1} + c_{3} + c_{5} + c_{7} \\
-c_{0} + c_{2} + c_{4} + c_{6} \\
c_{1} + c_{3} + c_{5} + c_{7} \\
c_{0} + c_{2} + c_{4} + c_{6} \\
c_{1} + c_{3} + c_{5} - c_{7}
\end{pmatrix}$ \\ 
\hline 
$(\zeta, \zeta^{5})$ & $\begin{pmatrix}
0 & 0 & 0 & 1 & 0 & 0 & 0 & -1 & c_{0} \\
1 & 0 & 0 & 0 & 1 & 0 & 0 & 0 & c_{1} \\
0 & 1 & 0 & 0 & 0 & 1 & 0 & 0 & c_{2} \\
0 & 0 & 1 & 0 & 0 & 0 & 1 & 0 & c_{3} \\
0 & 0 & 0 & 1 & 0 & 0 & 0 & 1 & c_{4} \\
-1 & 0 & 0 & 0 & 1 & 0 & 0 & 0 & c_{5} \\
0 & -1 & 0 & 0 & 0 & 1 & 0 & 0 & c_{6} \\
0 & 0 & -1 & 0 & 0 & 0 & 1 & 0 & c_{7}
\end{pmatrix}$ & $16$ & $\frac{1}{2} \begin{pmatrix}
c_{1} - c_{5} \\
c_{2} - c_{6} \\
c_{3} - c_{7} \\
c_{0} + c_{4} \\
c_{1} + c_{5} \\
c_{2} + c_{6} \\
c_{3} + c_{7} \\
-c_{0} + c_{4}
\end{pmatrix}$ \\ 
\hline 
$(\zeta, \zeta^{7})$ & $\begin{pmatrix}
0 & 1 & 0 & 0 & 0 & 0 & 0 & -1 & c_{0} \\
1 & 0 & 1 & 0 & 0 & 0 & 0 & 0 & c_{1} \\
0 & 1 & 0 & 1 & 0 & 0 & 0 & 0 & c_{2} \\
0 & 0 & 1 & 0 & 1 & 0 & 0 & 0 & c_{3} \\
0 & 0 & 0 & 1 & 0 & 1 & 0 & 0 & c_{4} \\
0 & 0 & 0 & 0 & 1 & 0 & 1 & 0 & c_{5} \\
0 & 0 & 0 & 0 & 0 & 1 & 0 & 1 & c_{6} \\
-1 & 0 & 0 & 0 & 0 & 0 & 1 & 0 & c_{7}
\end{pmatrix}$ & $4$ & $\frac{1}{2} \begin{pmatrix}
c_{1} - c_{3} + c_{5} - c_{7} \\
c_{0} + c_{2} - c_{4} + c_{6} \\
c_{1} + c_{3} - c_{5} + c_{7} \\
-c_{0} + c_{2} + c_{4} - c_{6} \\
-c_{1} + c_{3} + c_{5} - c_{7} \\
c_{0} - c_{2} + c_{4} + c_{6} \\
c_{1} - c_{3} + c_{5} + c_{7} \\
-c_{0} + c_{2} - c_{4} + c_{6}
\end{pmatrix}$ \\ 
\hline 
$(\zeta, \zeta^{9})$ & $\begin{pmatrix}
0 & 0 & 0 & 0 & 0 & 0 & 0 & -2 & c_{0} \\
2 & 0 & 0 & 0 & 0 & 0 & 0 & 0 & c_{1} \\
0 & 2 & 0 & 0 & 0 & 0 & 0 & 0 & c_{2} \\
0 & 0 & 2 & 0 & 0 & 0 & 0 & 0 & c_{3} \\
0 & 0 & 0 & 2 & 0 & 0 & 0 & 0 & c_{4} \\
0 & 0 & 0 & 0 & 2 & 0 & 0 & 0 & c_{5} \\
0 & 0 & 0 & 0 & 0 & 2 & 0 & 0 & c_{6} \\
0 & 0 & 0 & 0 & 0 & 0 & 2 & 0 & c_{7}
\end{pmatrix}$ & $256$ & $\frac{1}{2} \begin{pmatrix}
c_{1} \\
c_{2} \\
c_{3} \\
c_{4} \\
c_{5} \\
c_{6} \\
c_{7} \\
-c_{0}
\end{pmatrix}$ \\ 
\hline 
$(\zeta, \zeta^{11})$ & $\begin{pmatrix}
0 & 0 & 0 & 0 & 0 & -1 & 0 & -1 & c_{0} \\
1 & 0 & 0 & 0 & 0 & 0 & -1 & 0 & c_{1} \\
0 & 1 & 0 & 0 & 0 & 0 & 0 & -1 & c_{2} \\
1 & 0 & 1 & 0 & 0 & 0 & 0 & 0 & c_{3} \\
0 & 1 & 0 & 1 & 0 & 0 & 0 & 0 & c_{4} \\
0 & 0 & 1 & 0 & 1 & 0 & 0 & 0 & c_{5} \\
0 & 0 & 0 & 1 & 0 & 1 & 0 & 0 & c_{6} \\
0 & 0 & 0 & 0 & 1 & 0 & 1 & 0 & c_{7}
\end{pmatrix}$ & $4$ & $\frac{1}{2} \begin{pmatrix}
c_{1} + c_{3} - c_{5} + c_{7} \\
-c_{0} + c_{2} + c_{4} - c_{6} \\
-c_{1} + c_{3} + c_{5} - c_{7} \\
c_{0} - c_{2} + c_{4} + c_{6} \\
c_{1} - c_{3} + c_{5} + c_{7} \\
-c_{0} + c_{2} - c_{4} + c_{6} \\
-c_{1} + c_{3} - c_{5} + c_{7} \\
-c_{0} - c_{2} + c_{4} - c_{6}
\end{pmatrix}$ \\ 
\hline 
$(\zeta, \zeta^{13})$ & $\begin{pmatrix}
0 & 0 & 0 & -1 & 0 & 0 & 0 & -1 & c_{0} \\
1 & 0 & 0 & 0 & -1 & 0 & 0 & 0 & c_{1} \\
0 & 1 & 0 & 0 & 0 & -1 & 0 & 0 & c_{2} \\
0 & 0 & 1 & 0 & 0 & 0 & -1 & 0 & c_{3} \\
0 & 0 & 0 & 1 & 0 & 0 & 0 & -1 & c_{4} \\
1 & 0 & 0 & 0 & 1 & 0 & 0 & 0 & c_{5} \\
0 & 1 & 0 & 0 & 0 & 1 & 0 & 0 & c_{6} \\
0 & 0 & 1 & 0 & 0 & 0 & 1 & 0 & c_{7}
\end{pmatrix}$ & $16$ & $\frac{1}{2} \begin{pmatrix}
c_{1} + c_{5} \\
c_{2} + c_{6} \\
c_{3} + c_{7} \\
-c_{0} + c_{4} \\
-c_{1} + c_{5} \\
-c_{2} + c_{6} \\
-c_{3} + c_{7} \\
-c_{0} - c_{4}
\end{pmatrix}$ \\ 
\hline 
$(\zeta, \zeta^{15})$ & $\begin{pmatrix}
0 & -1 & 0 & 0 & 0 & 0 & 0 & -1 & c_{0} \\
1 & 0 & -1 & 0 & 0 & 0 & 0 & 0 & c_{1} \\
0 & 1 & 0 & -1 & 0 & 0 & 0 & 0 & c_{2} \\
0 & 0 & 1 & 0 & -1 & 0 & 0 & 0 & c_{3} \\
0 & 0 & 0 & 1 & 0 & -1 & 0 & 0 & c_{4} \\
0 & 0 & 0 & 0 & 1 & 0 & -1 & 0 & c_{5} \\
0 & 0 & 0 & 0 & 0 & 1 & 0 & -1 & c_{6} \\
1 & 0 & 0 & 0 & 0 & 0 & 1 & 0 & c_{7}
\end{pmatrix}$ & $4$ & $\frac{1}{2} \begin{pmatrix}
c_{1} + c_{3} + c_{5} + c_{7} \\
-c_{0} + c_{2} + c_{4} + c_{6} \\
-c_{1} + c_{3} + c_{5} + c_{7} \\
-c_{0} - c_{2} + c_{4} + c_{6} \\
-c_{1} - c_{3} + c_{5} + c_{7} \\
-c_{0} - c_{2} - c_{4} + c_{6} \\
-c_{1} - c_{3} - c_{5} + c_{7} \\
-c_{0} - c_{2} - c_{4} - c_{6}
\end{pmatrix}$ \\ 
\hline 
$(\zeta^{3}, \zeta^{5})$ & $\begin{pmatrix}
0 & 0 & 0 & 1 & 0 & -1 & 0 & 0 & c_{0} \\
0 & 0 & 0 & 0 & 1 & 0 & -1 & 0 & c_{1} \\
0 & 0 & 0 & 0 & 0 & 1 & 0 & -1 & c_{2} \\
1 & 0 & 0 & 0 & 0 & 0 & 1 & 0 & c_{3} \\
0 & 1 & 0 & 0 & 0 & 0 & 0 & 1 & c_{4} \\
-1 & 0 & 1 & 0 & 0 & 0 & 0 & 0 & c_{5} \\
0 & -1 & 0 & 1 & 0 & 0 & 0 & 0 & c_{6} \\
0 & 0 & -1 & 0 & 1 & 0 & 0 & 0 & c_{7} \\
\end{pmatrix}$ & $4$ & $\frac{1}{2} \begin{pmatrix}
c_{1} + c_{3} - c_{5} - c_{7} \\
c_{0} + c_{2} + c_{4} - c_{6} \\
c_{1} + c_{3} + c_{5} - c_{7} \\
c_{0} + c_{2} + c_{4} + c_{6} \\
c_{1} + c_{3} + c_{5} + c_{7} \\
-c_{0} + c_{2} + c_{4} + c_{6} \\
-c_{1} + c_{3} + c_{5} + c_{7} \\
-c_{0} - c_{2} + c_{4} + c_{6}
\end{pmatrix}$ \\ 
\hline 
$(\zeta^{3}, \zeta^{7})$ & $\begin{pmatrix}
0 & 1 & 0 & 0 & 0 & -1 & 0 & 0 & c_{0} \\
0 & 0 & 1 & 0 & 0 & 0 & -1 & 0 & c_{1} \\
0 & 0 & 0 & 1 & 0 & 0 & 0 & -1 & c_{2} \\
1 & 0 & 0 & 0 & 1 & 0 & 0 & 0 & c_{3} \\
0 & 1 & 0 & 0 & 0 & 1 & 0 & 0 & c_{4} \\
0 & 0 & 1 & 0 & 0 & 0 & 1 & 0 & c_{5} \\
0 & 0 & 0 & 1 & 0 & 0 & 0 & 1 & c_{6} \\
-1 & 0 & 0 & 0 & 1 & 0 & 0 & 0 & c_{7}
\end{pmatrix}$ & $16$ & $\frac{1}{2} \begin{pmatrix}
c_{3} - c_{7} \\
c_{0} + c_{4} \\
c_{1} + c_{5} \\
c_{2} + c_{6} \\
c_{3} + c_{7} \\
-c_{0} + c_{4} \\
-c_{1} + c_{5} \\
-c_{2} + c_{6}
\end{pmatrix}$ \\ 
\hline 
$(\zeta^{3}, \zeta^{9})$ & $\begin{pmatrix}
0 & 0 & 0 & 0 & 0 & -1 & 0 & -1 & c_{0} \\
1 & 0 & 0 & 0 & 0 & 0 & -1 & 0 & c_{1} \\
0 & 1 & 0 & 0 & 0 & 0 & 0 & -1 & c_{2} \\
1 & 0 & 1 & 0 & 0 & 0 & 0 & 0 & c_{3} \\
0 & 1 & 0 & 1 & 0 & 0 & 0 & 0 & c_{4} \\
0 & 0 & 1 & 0 & 1 & 0 & 0 & 0 & c_{5} \\
0 & 0 & 0 & 1 & 0 & 1 & 0 & 0 & c_{6} \\
0 & 0 & 0 & 0 & 1 & 0 & 1 & 0 & c_{7}
\end{pmatrix}$ & $4$ & $\frac{1}{2} \begin{pmatrix}
c_{1} + c_{3} - c_{5} + c_{7} \\
-c_{0} + c_{2} + c_{4} - c_{6} \\
-c_{1} + c_{3} + c_{5} - c_{7} \\
c_{0} - c_{2} + c_{4} + c_{6} \\
c_{1} - c_{3} + c_{5} + c_{7} \\
-c_{0} + c_{2} - c_{4} + c_{6} \\
-c_{1} + c_{3} - c_{5} + c_{7} \\
-c_{0} - c_{2} + c_{4} - c_{6}
\end{pmatrix}$ \\ 
\hline 
$(\zeta^{3}, \zeta^{11})$ & $\begin{pmatrix}
0 & 0 & 0 & 0 & 0 & -2 & 0 & 0 & c_{0} \\
0 & 0 & 0 & 0 & 0 & 0 & -2 & 0 & c_{1} \\
0 & 0 & 0 & 0 & 0 & 0 & 0 & -2 & c_{2} \\
2 & 0 & 0 & 0 & 0 & 0 & 0 & 0 & c_{3} \\
0 & 2 & 0 & 0 & 0 & 0 & 0 & 0 & c_{4} \\
0 & 0 & 2 & 0 & 0 & 0 & 0 & 0 & c_{5} \\
0 & 0 & 0 & 2 & 0 & 0 & 0 & 0 & c_{6} \\
0 & 0 & 0 & 0 & 2 & 0 & 0 & 0 & c_{7}
\end{pmatrix}$ & $256$ & $\frac{1}{2} \begin{pmatrix}
c_{3} \\
c_{4} \\
c_{5} \\
c_{6} \\
c_{7} \\
-c_{0} \\
-c_{1} \\
-c_{2}
\end{pmatrix}$ \\ 
\hline 
$(\zeta^{3}, \zeta^{13})$ & $\begin{pmatrix}
0 & 0 & 0 & -1 & 0 & -1 & 0 & 0 & c_{0} \\
0 & 0 & 0 & 0 & -1 & 0 & -1 & 0 & c_{1} \\
0 & 0 & 0 & 0 & 0 & -1 & 0 & -1 & c_{2} \\
1 & 0 & 0 & 0 & 0 & 0 & -1 & 0 & c_{3} \\
0 & 1 & 0 & 0 & 0 & 0 & 0 & -1 & c_{4} \\
1 & 0 & 1 & 0 & 0 & 0 & 0 & 0 & c_{5} \\
0 & 1 & 0 & 1 & 0 & 0 & 0 & 0 & c_{6} \\
0 & 0 & 1 & 0 & 1 & 0 & 0 & 0 & c_{7}
\end{pmatrix}$ & $4$ & $\frac{1}{2} \begin{pmatrix}
-c_{1} + c_{3} + c_{5} - c_{7} \\
c_{0} - c_{2} + c_{4} + c_{6} \\
c_{1} - c_{3} + c_{5} + c_{7} \\
-c_{0} + c_{2} - c_{4} + c_{6} \\
-c_{1} + c_{3} - c_{5} + c_{7} \\
-c_{0} - c_{2} + c_{4} - c_{6} \\
-c_{1} - c_{3} + c_{5} - c_{7} \\
c_{0} - c_{2} - c_{4} + c_{6}
\end{pmatrix}$ \\ 
\hline 
$(\zeta^{3}, \zeta^{15})$ & $\begin{pmatrix}
0 & -1 & 0 & 0 & 0 & -1 & 0 & 0 & c_{0} \\
0 & 0 & -1 & 0 & 0 & 0 & -1 & 0 & c_{1} \\
0 & 0 & 0 & -1 & 0 & 0 & 0 & -1 & c_{2} \\
1 & 0 & 0 & 0 & -1 & 0 & 0 & 0 & c_{3} \\
0 & 1 & 0 & 0 & 0 & -1 & 0 & 0 & c_{4} \\
0 & 0 & 1 & 0 & 0 & 0 & -1 & 0 & c_{5} \\
0 & 0 & 0 & 1 & 0 & 0 & 0 & -1 & c_{6} \\
1 & 0 & 0 & 0 & 1 & 0 & 0 & 0 & c_{7}
\end{pmatrix}$ & $16$ & $\frac{1}{2} \begin{pmatrix}
c_{3} + c_{7} \\
-c_{0} + c_{4} \\
-c_{1} + c_{5} \\
-c_{2} + c_{6} \\
-c_{3} + c_{7} \\
-c_{0} - c_{4} \\
-c_{1} - c_{5} \\
-c_{2} - c_{6}
\end{pmatrix}$ \\ 
\hline 
$(\zeta^{5}, \zeta^{7})$ & $\begin{pmatrix}
0 & 1 & 0 & -1 & 0 & 0 & 0 & 0 & c_{0} \\
0 & 0 & 1 & 0 & -1 & 0 & 0 & 0 & c_{1} \\
0 & 0 & 0 & 1 & 0 & -1 & 0 & 0 & c_{2} \\
0 & 0 & 0 & 0 & 1 & 0 & -1 & 0 & c_{3} \\
0 & 0 & 0 & 0 & 0 & 1 & 0 & -1 & c_{4} \\
1 & 0 & 0 & 0 & 0 & 0 & 1 & 0 & c_{5} \\
0 & 1 & 0 & 0 & 0 & 0 & 0 & 1 & c_{6} \\
-1 & 0 & 1 & 0 & 0 & 0 & 0 & 0 & c_{7}
\end{pmatrix}$ & $4$ & $\frac{1}{2} \begin{pmatrix}
c_{1} + c_{3} + c_{5} - c_{7} \\
c_{0} + c_{2} + c_{4} + c_{6} \\
c_{1} + c_{3} + c_{5} + c_{7} \\
-c_{0} + c_{2} + c_{4} + c_{6} \\
-c_{1} + c_{3} + c_{5} + c_{7} \\
-c_{0} - c_{2} + c_{4} + c_{6} \\
-c_{1} - c_{3} + c_{5} + c_{7} \\
-c_{0} - c_{2} - c_{4} + c_{6}
\end{pmatrix}$ \\ 
\hline 
$(\zeta^{5}, \zeta^{9})$ & $\begin{pmatrix}
0 & 0 & 0 & -1 & 0 & 0 & 0 & -1 & c_{0} \\
1 & 0 & 0 & 0 & -1 & 0 & 0 & 0 & c_{1} \\
0 & 1 & 0 & 0 & 0 & -1 & 0 & 0 & c_{2} \\
0 & 0 & 1 & 0 & 0 & 0 & -1 & 0 & c_{3} \\
0 & 0 & 0 & 1 & 0 & 0 & 0 & -1 & c_{4} \\
1 & 0 & 0 & 0 & 1 & 0 & 0 & 0 & c_{5} \\
0 & 1 & 0 & 0 & 0 & 1 & 0 & 0 & c_{6} \\
0 & 0 & 1 & 0 & 0 & 0 & 1 & 0 & c_{7}
\end{pmatrix}$ & $16$ & $\frac{1}{2} \begin{pmatrix}
c_{1} + c_{5} \\
c_{2} + c_{6} \\
c_{3} + c_{7} \\
-c_{0} + c_{4} \\
-c_{1} + c_{5} \\
-c_{2} + c_{6} \\
-c_{3} + c_{7} \\
-c_{0} - c_{4}
\end{pmatrix}$ \\ 
\hline 
$(\zeta^{5}, \zeta^{11})$ & $\begin{pmatrix}
0 & 0 & 0 & -1 & 0 & -1 & 0 & 0 & c_{0} \\
0 & 0 & 0 & 0 & -1 & 0 & -1 & 0 & c_{1} \\
0 & 0 & 0 & 0 & 0 & -1 & 0 & -1 & c_{2} \\
1 & 0 & 0 & 0 & 0 & 0 & -1 & 0 & c_{3} \\
0 & 1 & 0 & 0 & 0 & 0 & 0 & -1 & c_{4} \\
1 & 0 & 1 & 0 & 0 & 0 & 0 & 0 & c_{5} \\
0 & 1 & 0 & 1 & 0 & 0 & 0 & 0 & c_{6} \\
0 & 0 & 1 & 0 & 1 & 0 & 0 & 0 & c_{7}
\end{pmatrix}$ & $4$ & $\frac{1}{2} \begin{pmatrix}
-c_{1} + c_{3} + c_{5} - c_{7} \\
c_{0} - c_{2} + c_{4} + c_{6} \\
c_{1} - c_{3} + c_{5} + c_{7} \\
-c_{0} + c_{2} - c_{4} + c_{6} \\
-c_{1} + c_{3} - c_{5} + c_{7} \\
-c_{0} - c_{2} + c_{4} - c_{6} \\
-c_{1} - c_{3} + c_{5} - c_{7} \\
c_{0} - c_{2} - c_{4} + c_{6}
\end{pmatrix}$ \\ 
\hline 
$(\zeta^{5}, \zeta^{13})$ & $\begin{pmatrix}
0 & 0 & 0 & -2 & 0 & 0 & 0 & 0 & c_{0} \\
0 & 0 & 0 & 0 & -2 & 0 & 0 & 0 & c_{1} \\
0 & 0 & 0 & 0 & 0 & -2 & 0 & 0 & c_{2} \\
0 & 0 & 0 & 0 & 0 & 0 & -2 & 0 & c_{3} \\
0 & 0 & 0 & 0 & 0 & 0 & 0 & -2 & c_{4} \\
2 & 0 & 0 & 0 & 0 & 0 & 0 & 0 & c_{5} \\
0 & 2 & 0 & 0 & 0 & 0 & 0 & 0 & c_{6} \\
0 & 0 & 2 & 0 & 0 & 0 & 0 & 0 & c_{7}
\end{pmatrix}$ & $256$ & $\frac{1}{2} \begin{pmatrix}
c_{5} \\
c_{6} \\
c_{7} \\
-c_{0} \\
-c_{1} \\
-c_{2} \\
-c_{3} \\
-c_{4}
\end{pmatrix}$ \\ 
\hline 
$(\zeta^{5}, \zeta^{15})$ & $\begin{pmatrix}
0 & -1 & 0 & -1 & 0 & 0 & 0 & 0 & c_{0} \\
0 & 0 & -1 & 0 & -1 & 0 & 0 & 0 & c_{1} \\
0 & 0 & 0 & -1 & 0 & -1 & 0 & 0 & c_{2} \\
0 & 0 & 0 & 0 & -1 & 0 & -1 & 0 & c_{3} \\
0 & 0 & 0 & 0 & 0 & -1 & 0 & -1 & c_{4} \\
1 & 0 & 0 & 0 & 0 & 0 & -1 & 0 & c_{5} \\
0 & 1 & 0 & 0 & 0 & 0 & 0 & -1 & c_{6} \\
1 & 0 & 1 & 0 & 0 & 0 & 0 & 0 & c_{7}
\end{pmatrix}$ & $4$ & $\frac{1}{2} \begin{pmatrix}
c_{1} - c_{3} + c_{5} + c_{7} \\
-c_{0} + c_{2} - c_{4} + c_{6} \\
-c_{1} + c_{3} - c_{5} + c_{7} \\
-c_{0} - c_{2} + c_{4} - c_{6} \\
-c_{1} - c_{3} + c_{5} - c_{7} \\
c_{0} - c_{2} - c_{4} + c_{6} \\
c_{1} - c_{3} - c_{5} + c_{7} \\
-c_{0} + c_{2} - c_{4} - c_{6}
\end{pmatrix}$ \\ 
\hline 
$(\zeta^{7}, \zeta^{9})$ & $\begin{pmatrix}
0 & -1 & 0 & 0 & 0 & 0 & 0 & -1 & c_{0} \\
1 & 0 & -1 & 0 & 0 & 0 & 0 & 0 & c_{1} \\
0 & 1 & 0 & -1 & 0 & 0 & 0 & 0 & c_{2} \\
0 & 0 & 1 & 0 & -1 & 0 & 0 & 0 & c_{3} \\
0 & 0 & 0 & 1 & 0 & -1 & 0 & 0 & c_{4} \\
0 & 0 & 0 & 0 & 1 & 0 & -1 & 0 & c_{5} \\
0 & 0 & 0 & 0 & 0 & 1 & 0 & -1 & c_{6} \\
1 & 0 & 0 & 0 & 0 & 0 & 1 & 0 & c_{7}
\end{pmatrix}$ & $4$ & $\frac{1}{2} \begin{pmatrix}
c_{1} + c_{3} + c_{5} + c_{7} \\
-c_{0} + c_{2} + c_{4} + c_{6} \\
-c_{1} + c_{3} + c_{5} + c_{7} \\
-c_{0} - c_{2} + c_{4} + c_{6} \\
-c_{1} - c_{3} + c_{5} + c_{7} \\
-c_{0} - c_{2} - c_{4} + c_{6} \\
-c_{1} - c_{3} - c_{5} + c_{7} \\
-c_{0} - c_{2} - c_{4} - c_{6}
\end{pmatrix}$ \\ 
\hline 
$(\zeta^{7}, \zeta^{11})$ & $\begin{pmatrix}
0 & -1 & 0 & 0 & 0 & -1 & 0 & 0 & c_{0} \\
0 & 0 & -1 & 0 & 0 & 0 & -1 & 0 & c_{1} \\
0 & 0 & 0 & -1 & 0 & 0 & 0 & -1 & c_{2} \\
1 & 0 & 0 & 0 & -1 & 0 & 0 & 0 & c_{3} \\
0 & 1 & 0 & 0 & 0 & -1 & 0 & 0 & c_{4} \\
0 & 0 & 1 & 0 & 0 & 0 & -1 & 0 & c_{5} \\
0 & 0 & 0 & 1 & 0 & 0 & 0 & -1 & c_{6} \\
1 & 0 & 0 & 0 & 1 & 0 & 0 & 0 & c_{7}
\end{pmatrix}$ & $16$ & $\frac{1}{2} \begin{pmatrix}
c_{3} + c_{7} \\
-c_{0} + c_{4} \\
-c_{1} + c_{5} \\
-c_{2} + c_{6} \\
-c_{3} + c_{7} \\
-c_{0} - c_{4} \\
-c_{1} - c_{5} \\
-c_{2} - c_{6}
\end{pmatrix}$ \\ 
\hline 
$(\zeta^{7}, \zeta^{13})$ & $\begin{pmatrix}
0 & -1 & 0 & -1 & 0 & 0 & 0 & 0 & c_{0} \\
0 & 0 & -1 & 0 & -1 & 0 & 0 & 0 & c_{1} \\
0 & 0 & 0 & -1 & 0 & -1 & 0 & 0 & c_{2} \\
0 & 0 & 0 & 0 & -1 & 0 & -1 & 0 & c_{3} \\
0 & 0 & 0 & 0 & 0 & -1 & 0 & -1 & c_{4} \\
1 & 0 & 0 & 0 & 0 & 0 & -1 & 0 & c_{5} \\
0 & 1 & 0 & 0 & 0 & 0 & 0 & -1 & c_{6} \\
1 & 0 & 1 & 0 & 0 & 0 & 0 & 0 & c_{7} \\
\end{pmatrix}$ & $4$ & $\frac{1}{2} \begin{pmatrix}
c_{1} - c_{3} + c_{5} + c_{7} \\
-c_{0} + c_{2} - c_{4} + c_{6} \\
-c_{1} + c_{3} - c_{5} + c_{7} \\
-c_{0} - c_{2} + c_{4} - c_{6} \\
-c_{1} - c_{3} + c_{5} - c_{7} \\
c_{0} - c_{2} - c_{4} + c_{6} \\
c_{1} - c_{3} - c_{5} + c_{7} \\
-c_{0} + c_{2} - c_{4} - c_{6}
\end{pmatrix}$ \\ 
\hline 
$(\zeta^{7}, \zeta^{15})$ & $\begin{pmatrix}
0 & -2 & 0 & 0 & 0 & 0 & 0 & 0 & c_{0} \\
0 & 0 & -2 & 0 & 0 & 0 & 0 & 0 & c_{1} \\
0 & 0 & 0 & -2 & 0 & 0 & 0 & 0 & c_{2} \\
0 & 0 & 0 & 0 & -2 & 0 & 0 & 0 & c_{3} \\
0 & 0 & 0 & 0 & 0 & -2 & 0 & 0 & c_{4} \\
0 & 0 & 0 & 0 & 0 & 0 & -2 & 0 & c_{5} \\
0 & 0 & 0 & 0 & 0 & 0 & 0 & -2 & c_{6} \\
2 & 0 & 0 & 0 & 0 & 0 & 0 & 0 & c_{7}
\end{pmatrix}$ & $256$ & $\frac{1}{2} \begin{pmatrix}
c_{7} \\
-c_{0} \\
-c_{1} \\
-c_{2} \\
-c_{3} \\
-c_{4} \\
-c_{5} \\
-c_{6}
\end{pmatrix}$ \\ 
\hline 
$(\zeta^{9}, \zeta^{11})$ & $\begin{pmatrix}
0 & 0 & 0 & 0 & 0 & -1 & 0 & 1 & c_{0} \\
-1 & 0 & 0 & 0 & 0 & 0 & -1 & 0 & c_{1} \\
0 & -1 & 0 & 0 & 0 & 0 & 0 & -1 & c_{2} \\
1 & 0 & -1 & 0 & 0 & 0 & 0 & 0 & c_{3} \\
0 & 1 & 0 & -1 & 0 & 0 & 0 & 0 & c_{4} \\
0 & 0 & 1 & 0 & -1 & 0 & 0 & 0 & c_{5} \\
0 & 0 & 0 & 1 & 0 & -1 & 0 & 0 & c_{6} \\
0 & 0 & 0 & 0 & 1 & 0 & -1 & 0 & c_{7}
\end{pmatrix}$ & $4$ & $\frac{1}{2} \begin{pmatrix}
-c_{1} + c_{3} + c_{5} + c_{7} \\
-c_{0} - c_{2} + c_{4} + c_{6} \\
-c_{1} - c_{3} + c_{5} + c_{7} \\
-c_{0} - c_{2} - c_{4} + c_{6} \\
-c_{1} - c_{3} - c_{5} + c_{7} \\
-c_{0} - c_{2} - c_{4} - c_{6} \\
-c_{1} - c_{3} - c_{5} - c_{7} \\
c_{0} - c_{2} - c_{4} - c_{6}
\end{pmatrix}$ \\ 
\hline 
$(\zeta^{9}, \zeta^{13})$ & $\begin{pmatrix}
0 & 0 & 0 & -1 & 0 & 0 & 0 & 1 & c_{0} \\
-1 & 0 & 0 & 0 & -1 & 0 & 0 & 0 & c_{1} \\
0 & -1 & 0 & 0 & 0 & -1 & 0 & 0 & c_{2} \\
0 & 0 & -1 & 0 & 0 & 0 & -1 & 0 & c_{3} \\
0 & 0 & 0 & -1 & 0 & 0 & 0 & -1 & c_{4} \\
1 & 0 & 0 & 0 & -1 & 0 & 0 & 0 & c_{5} \\
0 & 1 & 0 & 0 & 0 & -1 & 0 & 0 & c_{6} \\
0 & 0 & 1 & 0 & 0 & 0 & -1 & 0 & c_{7}
\end{pmatrix}$ & $16$ & $\frac{1}{2} \begin{pmatrix}
-c_{1} + c_{5} \\
-c_{2} + c_{6} \\
-c_{3} + c_{7} \\
-c_{0} - c_{4} \\
-c_{1} - c_{5} \\
-c_{2} - c_{6} \\
-c_{3} - c_{7} \\
c_{0} - c_{4}
\end{pmatrix}$ \\ 
\hline 
$(\zeta^{9}, \zeta^{15})$ & $\begin{pmatrix}
0 & -1 & 0 & 0 & 0 & 0 & 0 & 1 & c_{0} \\
-1 & 0 & -1 & 0 & 0 & 0 & 0 & 0 & c_{1} \\
0 & -1 & 0 & -1 & 0 & 0 & 0 & 0 & c_{2} \\
0 & 0 & -1 & 0 & -1 & 0 & 0 & 0 & c_{3} \\
0 & 0 & 0 & -1 & 0 & -1 & 0 & 0 & c_{4} \\
0 & 0 & 0 & 0 & -1 & 0 & -1 & 0 & c_{5} \\
0 & 0 & 0 & 0 & 0 & -1 & 0 & -1 & c_{6} \\
1 & 0 & 0 & 0 & 0 & 0 & -1 & 0 & c_{7}
\end{pmatrix}$ & $4$ & $\frac{1}{2} \begin{pmatrix}
-c_{1} + c_{3} - c_{5} + c_{7} \\
-c_{0} - c_{2} + c_{4} - c_{6} \\
-c_{1} - c_{3} + c_{5} - c_{7} \\
c_{0} - c_{2} - c_{4} + c_{6} \\
c_{1} - c_{3} - c_{5} + c_{7} \\
-c_{0} + c_{2} - c_{4} - c_{6} \\
-c_{1} + c_{3} - c_{5} - c_{7} \\
c_{0} - c_{2} + c_{4} - c_{6}
\end{pmatrix}$ \\ 
\hline 
$(\zeta^{11}, \zeta^{13})$ & $\begin{pmatrix}
0 & 0 & 0 & -1 & & 1 & 0 & 0 & c_{0} \\
0 & 0 & 0 & 0 & -1 & 0 & 1 & 0 & c_{1} \\
0 & 0 & 0 & 0 & 0 & -1 & 0 & 1 & c_{2} \\
-1 & 0 & 0 & 0 & 0 & 0 & -1 & 0 & c_{3} \\
0 & -1 & 0 & 0 & 0 & 0 & 0 & -1 & c_{4} \\
1 & 0 & -1 & 0 & 0 & 0 & 0 & 0 & c_{5} \\
0 & 1 & 0 & -1 & 0 & 0 & 0 & 0 & c_{6} \\
0 & 0 & 1 & 0 & -1 & 0 & 0 & 0 & c_{7}
\end{pmatrix}$ & $4$ & $\frac{1}{2} \begin{pmatrix}
-c_{1} - c_{3} + c_{5} + c_{7} \\
-c_{0} - c_{2} - c_{4} + c_{6} \\
-c_{1} - c_{3} - c_{5} + c_{7} \\
-c_{0} - c_{2} - c_{4} - c_{6} \\
-c_{1} - c_{3} - c_{5} - c_{7} \\
c_{0} - c_{2} - c_{4} - c_{6} \\
c_{1} - c_{3} - c_{5} - c_{7} \\
c_{0} + c_{2} - c_{4} - c_{6}
\end{pmatrix}$ \\ 
\hline 
$(\zeta^{11}, \zeta^{15})$ & $\begin{pmatrix}
0 & -1 & 0 & 0 & 0 & 1 & 0 & 0 & c_{0} \\
0 & 0 & -1 & 0 & 0 & 0 & 1 & 0 & c_{1} \\
0 & 0 & 0 & -1 & 0 & 0 & 0 & 1 & c_{2} \\
-1 & 0 & 0 & 0 & -1 & 0 & 0 & 0 & c_{3} \\
0 & -1 & 0 & 0 & 0 & -1 & 0 & 0 & c_{4} \\
0 & 0 & -1 & 0 & 0 & 0 & -1 & 0 & c_{5} \\
0 & 0 & 0 & -1 & 0 & 0 & 0 & -1 & c_{6} \\
1 & 0 & 0 & 0 & -1 & 0 & 0 & 0 & c_{7}
\end{pmatrix}$ & $16$ & $\frac{1}{2} \begin{pmatrix}
-c_{3} + c_{7} \\
-c_{0} - c_{4} \\
-c_{1} - c_{5} \\
-c_{2} - c_{6} \\
-c_{3} - c_{7} \\
c_{0} - c_{4} \\
c_{1} - c_{5} \\
c_{2} - c_{6}
\end{pmatrix}$ \\ 
\hline 
$(\zeta^{13}, \zeta^{15})$ & $\begin{pmatrix}
0 & -1 & 0 & 1 & 0 & 0 & 0 & 0 & c_{0} \\
0 & 0 & -1 & 0 & 1 & 0 & 0 & 0 & c_{1} \\
0 & 0 & 0 & -1 & 0 & 1 & 0 & 0 & c_{2} \\
0 & 0 & 0 & 0 & -1 & 0 & 1 & 0 & c_{3} \\
0 & 0 & 0 & 0 & 0 & -1 & 0 & 1 & c_{4} \\
-1 & 0 & 0 & 0 & 0 & 0 & -1 & 0 & c_{5} \\
0 & -1 & 0 & 0 & 0 & 0 & 0 & -1 & c_{6} \\
1 & 0 & -1 & 0 & 0 & 0 & 0 & 0 & c_{7}
\end{pmatrix}$ & $4$ & $\frac{1}{2} \begin{pmatrix}
-c_{1} - c_{3} - c_{5} + c_{7} \\
-c_{0} - c_{2} - c_{4} - c_{6} \\
-c_{1} - c_{3} - c_{5} - c_{7} \\
c_{0} - c_{2} - c_{4} - c_{6} \\
c_{1} - c_{3} - c_{5} - c_{7} \\
c_{0} + c_{2} - c_{4} - c_{6} \\
c_{1} + c_{3} - c_{5} - c_{7} \\
c_{0} + c_{2} + c_{4} - c_{6}
\end{pmatrix}$ \\ 
\hline 
\end{longtable}

\small\addtolength{\tabcolsep}{-2pt}
\begin{longtable}{|c|c|c|c|}
\caption{\textbf{For} $\bm{(k,p) = (2,3),}$ $\bm{n = p^{k} = 9,}$ $\bm{\phi(9) = |\{1, 2, 4, 5, 7, 8\}| = 6,}$  $\bm{\Phi_{9}(x) = x^{6} + x^{3} + 1.}$}
\label{table 4} \\
\hline 
$(\sigma(\zeta), \tau(\zeta))$ & $(A ~ C)$ & $\textbf{det}(A)$ & $C'$ \\ 
\hline 
$(\zeta, \zeta^{2})$ & $\begin{pmatrix}
0 & 0 & 0 & 0 & 1 & -1 & c_{0} \\
1 & 0 & 0 & 0 & 0 & 1 & c_{1} \\
-1 & 1 & 0 & 0 & 0 & 0 & c_{2} \\
0 & -1 & 1 & 0 & 1 & -1 & c_{3} \\
0 & 0 & -1 & 1 & 0 & 1 & c_{4} \\
0 & 0 & 0 & -1 & 1 & 0 & c_{5}
\end{pmatrix}$ & $3$ & $\frac{1}{3} \begin{pmatrix}
2 c_{0} + 2 c_{1} - c_{2} - c_{3} - c_{4} - c_{5} \\
2 c_{0} + 2 c_{1} + 2 c_{2} - c_{3} - c_{4} - c_{5} \\
-c_{0} + 2 c_{1} + 2 c_{2} + 2 c_{3} - c_{4} - c_{5} \\
c_{0} + c_{1} + c_{2} + c_{3} + c_{4} - 2 c_{5} \\
c_{0} + c_{1} + c_{2} + c_{3} + c_{4} + c_{5} \\
-2 c_{0} + c_{1} + c_{2} + c_{3} + c_{4} + c_{5}
\end{pmatrix}$ \\ 
\hline 
$(\zeta, \zeta^{4})$ & $\begin{pmatrix}
0 & 0 & 1 & 0 & 0 & -2 & c_{0} \\
1 & 0 & 0 & 1 & 0 & 0 & c_{1} \\
0 & 1 & 0 & 0 & 1 & 0 & c_{2} \\
0 & 0 & 2 & 0 & 0 & -1 & c_{3} \\
-1 & 0 & 0 & 2 & 0 & 0 & c_{4} \\
0 & -1 & 0 & 0 & 2 & 0 & c_{5} 
\end{pmatrix}$ & $27$ & $\frac{1}{3} \begin{pmatrix}
2 c_{1} - c_{4} \\
2 c_{2} - c_{5} \\
-c_{0} + 2 c_{3} \\
c_{1} + c_{4} \\
c_{2} + c_{5} \\
-2 c_{0} + c_{3}
\end{pmatrix}$ \\  
\hline 
$(\zeta, \zeta^{5})$ & $\begin{pmatrix}
0 & 1 & 0 & 0 & -1 & -1 & c_{0} \\
1 & 0 & 1 & 0 & 0 & -1 & c_{1} \\
0 & 1 & 0 & 1 & 0 & 0 & c_{2} \\
0 & 1 & 1 & 0 & 0 & -1 & c_{3} \\
0 & 0 & 1 & 1 & 0 & 0 & c_{4} \\
-1 & 0 & 0 & 1 & 1 & 0 & c_{5}
\end{pmatrix}$ & $3$ & $\frac{1}{3} \begin{pmatrix}
-c_{0} + 2 c_{1} + 2 c_{2} - c_{3} - c_{4} - c_{5} \\
-c_{0} - c_{1} + 2 c_{2} + 2 c_{3} - c_{4} - c_{5} \\
-c_{0} - c_{1} - c_{2} + 2 c_{3} + 2 c_{4} - c_{5} \\
c_{0} + c_{1} + c_{2} - 2 c_{3} + c_{4} + c_{5} \\
-2 c_{0} + c_{1} + c_{2} + c_{3} - 2 c_{4} + c_{5} \\
-2 c_{0} - 2 c_{1} + c_{2} + c_{3} + c_{4} - 2 c_{5}
\end{pmatrix}$ \\ 
\hline 
$(\zeta, \zeta^{7})$ & $\begin{pmatrix}
0 & 0 & -1 & 0 & 0 & -1 & c_{0} \\
2 & 0 & 0 & -1 & 0 & 0 & c_{1} \\
0 & 2 & 0 & 0 & -1 & 0 & c_{2} \\
0 & 0 & 1 & 0 & 0 & -2 & c_{3} \\
1 & 0 & 0 & 1 & 0 & 0 & c_{4} \\
0 & 1 & 0 & 0 & 1 & 0 & c_{5}
\end{pmatrix}$ & $27$ & $\frac{1}{3} \begin{pmatrix}
c_{1} + c_{4} \\
c_{2} + c_{5} \\
-2 c_{0} + c_{3} \\
-c_{1} + 2 c_{4} \\
-c_{2} + 2 c_{5} \\
-c_{0} - c_{3}
\end{pmatrix}$ \\ 
\hline 
$(\zeta, \zeta^{8})$ & $\begin{pmatrix}
0 & -1 & 0 & 0 & 0 & -1 & c_{0} \\
1 & 0 & -1 & 0 & 0 & 0 & c_{1} \\
1 & 1 & 0 & -1 & 0 & 0 & c_{2} \\
0 & 0 & 1 & 0 & -1 & -1 & c_{3} \\
0 & 0 & 0 & 1 & 0 & -1 & c_{4} \\
1 & 0 & 0 & 0 & 1 & 0 & c_{5}
\end{pmatrix}$ & $3$ & $\frac{1}{3} \begin{pmatrix}
 -c_{0} + 2 c_{1} - c_{2} + 2 c_{3} - c_{4} + 2 c_{5} \\
-c_{0} - c_{1} + 2 c_{2} - c_{3} + 2 c_{4} - c_{5} \\
-c_{0} - c_{1} - c_{2} + 2 c_{3} - c_{4} + 2 c_{5} \\
-2 c_{0} + c_{1} - 2 c_{2} + c_{3} + c_{4} + c_{5} \\
c_{0} - 2 c_{1} + c_{2} - 2 c_{3} + c_{4} + c_{5} \\
-2 c_{0} + c_{1} - 2 c_{2} + c_{3} - 2 c_{4} + c_{5}
\end{pmatrix}$ \\ 
\hline 
$(\zeta^{2}, \zeta^{4})$ & $\begin{pmatrix}
0 & 0 & 1 & 0 & -1 & -1 & c_{0} \\
0 & 0 & 0 & 1 & 0 & -1 & c_{1} \\
1 & 0 & 0 & 0 & 1 & 0 & c_{2} \\
0 & 1 & 1 & 0 & -1 & 0 & c_{3} \\
-1 & 0 & 1 & 1 & 0 & -1 & c_{4} \\
0 & -1 & 0 & 1 & 1 & 0 & c_{5}
\end{pmatrix}$ & $3$ & $\frac{1}{3} \begin{pmatrix}
c_{0} + c_{1} + c_{2} + c_{3} - 2 c_{4} + c_{5} \\
-2 c_{0} + c_{1} + c_{2} + c_{3} + c_{4} - 2 c_{5} \\
c_{0} - 2 c_{1} + c_{2} + c_{3} + c_{4} + c_{5} \\
-c_{0} + 2 c_{1} - c_{2} + 2 c_{3} - c_{4} + 2 c_{5} \\
-c_{0} - c_{1} + 2 c_{2} - c_{3} + 2 c_{4} - c_{5} \\
-c_{0} - c_{1} - c_{2} + 2 c_{3} - c_{4} + 2 c_{5}
\end{pmatrix}$ \\ 
\hline 
$(\zeta^{2}, \zeta^{5})$ & $\begin{pmatrix}
0 & 1 & 0 & 0 & -2 & 0 & c_{0} \\
0 & 0 & 1 & 0 & 0 & -2 & c_{1} \\
1 & 0 & 0 & 1 & 0 & 0 & c_{2} \\
0 & 2 & 0 & 0 & -1 & 0 & c_{3} \\
0 & 0 & 2 & 0 & 0 & -1 & c_{4} \\
-1 & 0 & 0 & 2 & 0 & 0 & c_{5}
\end{pmatrix}$ & $27$ & $\frac{1}{3} \begin{pmatrix}
2 c_{2} - c_{5} \\
-c_{0} + 2 c_{3} \\
-c_{1} + 2 c_{4} \\
c_{2} + c_{5} \\
-2 c_{0} + c_{3} \\
-2 c_{1} + c_{4}
\end{pmatrix}$ \\ 
\hline 
$(\zeta^{2}, \zeta^{7})$ & $\begin{pmatrix}
0 & 0 & -1 & 0 & -1 & 0 & c_{0} \\
1 & 0 & 0 & -1 & 0 & -1 & c_{1} \\
1 & 1 & 0 & 0 & -1 & 0 & c_{2} \\
0 & 1 & 0 & 0 & -1 & -1 & c_{3} \\
1 & 0 & 1 & 0 & 0 & -1 & c_{4} \\
0 & 1 & 0 & 1 & 0 & 0 & c_{5}
\end{pmatrix}$ & $3$ & $\frac{1}{3} \begin{pmatrix}
c_{0} + c_{1} + c_{2} - 2 c_{3} + c_{4} + c_{5} \\
-2 c_{0} + c_{1} + c_{2} + c_{3} - 2 c_{4} + c_{5} \\
-2 c_{0} - 2 c_{1} + c_{2} + c_{3} + c_{4} - 2 c_{5} \\
2 c_{0} - c_{1} - c_{2} - c_{3} + 2 c_{4} + 2 c_{5} \\
-c_{0} + 2 c_{1} - c_{2} - c_{3} - c_{4} + 2 c_{5} \\
-c_{0} - c_{1} + 2 c_{2} - c_{3} - c_{4} - c_{5}
\end{pmatrix}$ \\ 
\hline 
$(\zeta^{2}, \zeta^{8})$ & $\begin{pmatrix}
0 & -1 & 0 & 0 & -1 & 0 & c_{0} \\
0 & 0 & -1 & 0 & 0 & -1 & c_{1} \\
2 & 0 & 0 & -1 & 0 & 0 & c_{2} \\
0 & 1 & 0 & 0 & -2 & 0 & c_{3} \\
0 & 0 & 1 & 0 & 0 & -2 & c_{4} \\
1 & 0 & 0 & 1 & 0 & 0 & c_{5}
\end{pmatrix}$ & $27$ & $\frac{1}{3} \begin{pmatrix}
c_{2} + c_{5} \\
-2 c_{0} + c_{3} \\
-2 c_{1} + c_{4} \\
-c_{2} + 2 c_{5} \\
-c_{0} - c_{3} \\
-c_{1} - c_{4}
\end{pmatrix}$ \\  
\hline 
$(\zeta^{4}, \zeta^{5})$ & $\begin{pmatrix}
0 & 1 & -1 & 0 & -1 & 1 & c_{0} \\
0 & 0 & 1 & -1 & 0 & -1 & c_{1} \\
0 & 0 & 0 & 1 & -1 & 0 & c_{2} \\
0 & 1 & -1 & 0 & 0 & 0 & c_{3} \\
1 & 0 & 1 & -1 & 0 & 0 & c_{4} \\
-1 & 1 & 0 & 1 & -1 & 0 & c_{5}
\end{pmatrix}$ & $3$ & $\frac{1}{3} \begin{pmatrix}
-c_{0} - c_{1} + 2 c_{2} + 2 c_{3} + 2 c_{4} - c_{5} \\
-c_{0} - c_{1} - c_{2} + 2 c_{3} + 2 c_{4} + 2 c_{5} \\
-c_{0} - c_{1} - c_{2} - c_{3} + 2 c_{4} + 2 c_{5} \\
-2 c_{0} - 2 c_{1} + c_{2} + c_{3} + c_{4} + c_{5} \\
-2 c_{0} - 2 c_{1} - 2 c_{2} + c_{3} + c_{4} + c_{5} \\
c_{0} - 2 c_{1} - 2 c_{2} - 2 c_{3} + c_{4} + c_{5}
\end{pmatrix}$ \\ 
\hline 
$(\zeta^{4}, \zeta^{7})$ & $\begin{pmatrix}
0 & 0 & -2 & 0 & 0 & 1 & c_{0} \\
1 & 0 & 0 & -2 & 0 & 0 & c_{1} \\
0 & 1 & 0 & 0 & -2 & 0 & c_{2} \\
0 & 0 & -1 & 0 & 0 & -1 & c_{3} \\
2 & 0 & 0 & -1 & 0 & 0 & c_{4} \\
0 & 2 & 0 & 0 & -1 & 0 & c_{5} 
\end{pmatrix}$ & $27$ & $\frac{1}{3} \begin{pmatrix}
-c_{1} + 2 c_{4} \\
-c_{2} + 2 c_{5} \\
-c_{0} - c_{3} \\
-2 c_{1} + c_{4} \\
-2 c_{2} + c_{5} \\
c_{0} - 2 c_{3}
\end{pmatrix}$ \\ 
\hline 
$(\zeta^{4}, \zeta^{8})$ & $\begin{pmatrix}
0 & -1 & -1 & 0 & 0 & 1 & c_{0} \\
0 & 0 & -1 & -1 & 0 & 0 & c_{1} \\
1 & 0 & 0 & -1 & -1 & 0 & c_{2} \\
0 & 0 & -1 & 0 & -1 & 0 & c_{3} \\
1 & 0 & 0 & -1 & 0 & -1 & c_{4} \\
1 & 1 & 0 & 0 & -1 & 0 & c_{5}
\end{pmatrix}$ & $3$ & $\frac{1}{3} \begin{pmatrix}
2 c_{0} - c_{1} - c_{2} - c_{3} + 2 c_{4} + 2 c_{5} \\
- c_{0} + 2 c_{1} - c_{2} - c_{3} - c_{4} + 2 c_{5} \\
-c_{0} - c_{1} + 2 c_{2} - c_{3} - c_{4} - c_{5} \\
c_{0} - 2 c_{1} - 2 c_{2} + c_{3} + c_{4} + c_{5} \\
c_{0} + c_{1} - 2 c_{2} - 2 c_{3} + c_{4} + c_{5} \\
c_{0} + c_{1} + c_{2} - 2 c_{3} - 2 c_{4} + c_{5}
\end{pmatrix}$ \\ 
\hline 
$(\zeta^{5}, \zeta^{7})$ & $\begin{pmatrix}
0 & -1 & -1 & 0 & 1 & 0 & c_{0} \\
1 & 0 & -1 & -1 & 0 & 1 & c_{1} \\
0 & 1 & 0 & -1 & -1 & 0 & c_{2} \\
0 & -1 & 0 & 0 & 0 & -1 & c_{3} \\
1 & 0 & -1 & 0 & 0 & 0 & c_{4} \\
1 & 1 & 0 & -1 & 0 & 0 & c_{5}
\end{pmatrix}$ & $3$ & $\frac{1}{3} \begin{pmatrix}
-2 c_{0} + c_{1} - 2 c_{2} + c_{3} + c_{4} + c_{5} \\
c_{0} - 2 c_{1} + c_{2} - 2 c_{3} + c_{4} + c_{5} \\
-2 c_{0} + c_{1} - 2 c_{2} + c_{3} - 2 c_{4} + c_{5} \\
-c_{0} - c_{1} - c_{2} - c_{3} + 2 c_{4} - c_{5} \\
2 c_{0} - c_{1} - c_{2} - c_{3} - c_{4} + 2 c_{5} \\
-c_{0} + 2 c_{1} - c_{2} - c_{3} - c_{4} - c_{5}
\end{pmatrix}$ \\ 
\hline 
$(\zeta^{5}, \zeta^{8})$ & $\begin{pmatrix}
0 & -2 & 0 & 0 & 1 & 0 & c_{0} \\
0 & 0 & -2 & 0 & 0 & 1 & c_{1} \\
1 & 0 & 0 & -2 & 0 & 0 & c_{2} \\
0 & -1 & 0 & 0 & -1 & 0 & c_{3} \\
0 & 0 & -1 & 0 & 0 & -1 & c_{4} \\
2 & 0 & 0 & -1 & 0 & 0 & c_{5}
\end{pmatrix}$ & $27$ & $\frac{1}{3} \begin{pmatrix}
-c_{2} + 2 c_{5} \\
-c_{0} - c_{3} \\
-c_{1} - c_{4} \\
-2 c_{2} + c_{5} \\
c_{0} - 2 c_{3} \\
c_{1} - 2 c_{4}
\end{pmatrix}$ \\ 
\hline 
$(\zeta^{7}, \zeta^{8})$ & $\begin{pmatrix}
0 & -1 & 1 & 0 & 0 & 0 & c_{0} \\
-1 & 0 & -1 & 1 & 0 & 0 & c_{1} \\
1 & -1 & 0 & -1 & 1 & 0 & c_{2} \\
0 & 0 & 0 & 0 & -1 & 1 & c_{3} \\
-1 & 0 & 0 & 0 & 0 & -1 & c_{4} \\
1 & -1 & 0 & 0 & 0 & 0 & c_{5}
\end{pmatrix}$ & $3$ & $\frac{1}{3} \begin{pmatrix}
-c_{0} - c_{1} - c_{2} - c_{3} - c_{4} + 2 c_{5} \\
-c_{0} - c_{1} - c_{2} - c_{3} - c_{4} - c_{5} \\
2 c_{0} - c_{1} - c_{2} - c_{3} - c_{4} - c_{5} \\
c_{0} + c_{1} - 2 c_{2} - 2 c_{3} - 2 c_{4} + c_{5} \\
c_{0} + c_{1} + c_{2} - 2 c_{3} - 2 c_{4} - 2 c_{5} \\
c_{0} + c_{1} + c_{2} + c_{3} - 2 c_{4} - 2 c_{5}
\end{pmatrix}$ \\ 
\hline 
\end{longtable}

\section{Outer $(\sigma, \tau)$-Derivations of the Ring of Algebraic Integers of Cyclotomic Number Fields}\label{section 5}
In this section, we assume the same notations as in Section \ref{section 4}. That is, $\mathbb{K} = \mathbb{Q}(\zeta)$ is an $n^{\text{th}}$-cyclotomic number field, and $O_{\mathbb{K}} = \mathbb{Z}[\zeta]$ is its ring of algebraic integers, and $\sigma, \tau:O_{\mathbb{K}} \rightarrow O_{\mathbb{K}}$ with $\sigma(\zeta) = \zeta^{u}$ and $\tau(\zeta) = \zeta^{v}$ for some $u, v \in U(n)$ with $u \neq v$, are two different non-zero $\mathbb{Z}$-algebra endomorphisms.  

Finding the non-zero outer derivations in rings and algebras has been considered an important problem. Based on the findings of sections \ref{section 3} and \ref{section 4}, in this section, we conjecture a necessary and sufficient condition for a $(\sigma, \tau)$-derivation $D: O_{\mathbb{K}} \rightarrow O_{\mathbb{K}}$ to be outer. Again, we do that for two different forms of $n$, namely, $n=2^{r}p$ ($r \in \mathbb{N}$) for an odd rational prime $p$ and $n = p^{k}$ ($k \in \mathbb{N} \setminus \{1\}$) for any rational prime $p$. In this way, we conjecture the existence and non-existence of non-zero outer twisted derivations of $O_{\mathbb{K}}$ and hence conjecture a solution of the twisted derivation problem for $O_{\mathbb{K}}$ for these two forms of $n$.

\subsection{For $n=2^{r}p$, $r \in \mathbb{N}$ and $p$ an odd rational prime}\label{subsection 5.1}
The following \th\ref{conjecture 5.1} gives the existence and non-existence of non-zero outer $(\sigma, \tau)$-derivations of $O_{\mathbb{K}}$, when $n$ has the form $n=2^{r}p$, where $r \in \mathbb{N}$ and $p$ is an odd rational prime.

\begin{conjecture}\th\label{conjecture 5.1}
Under the hypotheses of \th\ref{conjecture 4.1}, the following statements hold for a $(\sigma, \tau)$-derivation $D:O_{\mathbb{K}} \rightarrow O_{\mathbb{K}}$ with $D(\zeta) = \sum_{i=0}^{2^{r-1}(p-1)-1} c_{i} \zeta^{i} \in O_{\mathbb{K}}$.
\begin{enumerate}
\item[(i)] If $1 \leq e_{1} \leq r-1$ and $e_{2} \geq 1$, then $D$ is outer if and only if $\frac{1}{2^{2^{i}(p-1)}}(Adj(A)C) \notin \mathbb{Z}^{2^{r}(p-1)}$. In particular, $\text{Inn}_{(\sigma, \tau)}(O_{\mathbb{K}}) \subsetneq \mathcal{D}_{(\sigma, \tau)}(O_{\mathbb{K}})$, that is, $O_{\mathbb{K}}$ has non-zero outer $(\sigma, \tau)$-derivations.
\item[(ii)] If $e_{1} \geq r$ and $e_{2} = 0$, then $D$ is outer if and only if $\frac{1}{p^{2^{r-1}}}(Adj(A)C) \notin \mathbb{Z}^{2^{r}(p-1)}$. In particular, $\text{Inn}_{(\sigma, \tau)}(O_{\mathbb{K}}) \subsetneq \mathcal{D}_{(\sigma, \tau)}(O_{\mathbb{K}})$, that is, $O_{\mathbb{K}}$ has non-zero outer $(\sigma, \tau)$-derivations.
\item[(iii)] In all other cases, $D$ is always inner. In particular, $\mathcal{D}_{(\sigma, \tau)}(O_{\mathbb{K}}) = \text{Inn}_{(\sigma, \tau)}(O_{\mathbb{K}})$, that is, $O_{\mathbb{K}}$ has no non-zero outer $(\sigma, \tau)$-derivations, that is, all $(\sigma, \tau)$-derivations of $O_{\mathbb{K}}$ are inner.
\end{enumerate}
\end{conjecture} 

The above \th\ref{conjecture 5.1} is formed in view of Remark \ref{remark}, \th\ref{corollary 3.7}, and \th\ref{conjecture 4.2}.

\subsection{For $n = p^{k}$, $k \in \mathbb{N} \setminus \{1\}$ and $p$ a rational prime}\label{subsection 5.2}
The following \th\ref{conjecture 5.2} gives the existence and non-existence of non-zero outer $(\sigma, \tau)$-derivations of $O_{\mathbb{K}}$, when $n$ has the form $n=p^{k}$, where $k \in \mathbb{N} \setminus \{1\}$ and $p$ is any rational prime.

\begin{conjecture}\th\label{conjecture 5.2}
Under the hypotheses of \th\ref{conjecture 4.3}, a $(\sigma, \tau)$-derivation $D:O_{\mathbb{K}} \rightarrow O_{\mathbb{K}}$ with $D(\zeta) = \sum_{i=0}^{p^{k-1}(p-1)-1} c_{i} \zeta^{i}$, is outer if and only if $\frac{1}{p^{e_{1}}}(Adj(A)C) \notin \mathbb{Z}^{p^{k-1}(p-1)}$. In particular, $\text{Inn}_{(\sigma, \tau)}(O_{\mathbb{K}}) \subsetneq \mathcal{D}_{(\sigma, \tau)}(O_{\mathbb{K}})$, that is, $O_{\mathbb{K}}$ has non-zero outer derivations.
\end{conjecture}

The above \th\ref{conjecture 5.2} is formed in view of Remark \ref{remark}, \th\ref{corollary 3.7}, and \th\ref{conjecture 4.4}.

\section{Application of Theorem \ref{theorem 3.2} in the Construction of Codes}\label{section 6}
In Section 4 of \cite{Manju2023b}, the authors have introduced the notion of Hom-IDD codes. In this section, we construct some binary Hom-IDD codes, thus giving the application of our main Theorem \ref{theorem 3.2} in the construction of codes in coding theory.

\begin{example}
Let $\mathbb{K} = \mathbb{Q}(\zeta)$, $\zeta$ primitive $n^{\text{th}}$-root of unity, where $n = 4p$ and $p=7$. Let $\sigma(\zeta) = \zeta$ and $\tau(\zeta) = \zeta^{3}$. Then by Theorem \ref{theorem 3.2} (in particular, Corollary \ref{corollary 3.7}), the map $D:O_{\mathbb{K}} \rightarrow O_{\mathbb{K}}$ given by $$D(\zeta) = 1 + \zeta + \zeta^{2} + \zeta^{5} + \zeta^{6} +  \zeta^{9} + \zeta^{11} $$ is a $(\sigma, \tau)$-derivation of $O_{\mathbb{K}} = \mathbb{Z}[\zeta]$. Also, $[\mathbb{K} : \mathbb{Q}] = \phi(n) = \phi(28) = 12$ and the $n^{\text{th}}$ cyclotomic polynomial is $\Phi_{n}(x) = \Phi_{28}(x) = x^{12} - x^{10} + x^{8} - x^{6} + x^{4} - x^{2} + 1$.

Below, we have obtained binary Hom-IDD codes of length $\phi(n) = \phi(28) = 12$ over the finite field $\mathbb{Z}_{2}$ with various parameters. Note that in the table below, $n'$ denotes the length of the code, $k'$ denotes the dimension of the code, and $d'$ denotes the distance of the code.

{\centering
\begin{longtable}{|c|c|c|c|}
\hline 
\textbf{Basis $S_{i}$} & \thead{\textbf{Code Description:}\\ $[n',k',d']$} & \textbf{Code Properties} & \thead{\textbf{Dual Code} \\ \textbf{Description:} \\ $[n',k',d']$} \\ 
\hline \hline
$S_{1}$ & $[12,6,3]$ & LCD & $[12,6,2]$ \\
\hline
$S_{2}$ & $[12,7,3]$ & LCD & $[12,5,3]$ \\
\hline
$S_{3}$ & $[12,5,4]$ & non-LCD \& Optimal & $[12,7,2]$ \\
\hline
$S_{4}$ & $[12,3,5]$ & LCD & \thead{$[12,9,2]$ \\ (optimal)} \\
\hline
$S_{5}$ & $[12,4,4]$ & LCD & $[12,8,2]$ \\
\hline
$S_{6}$ & $[12,4,5]$ & LCD & $[12,8,2]$ \\
\hline
$S_{7}$ & $[12,6,3]$ & non-LCD & $[12,6,3]$ \\
\hline
$S_{8}$ & $[12,8,2]$ & LCD & $[12,4,4]$ \\
\hline
$S_{9}$ & $[12,9,2]$ & non-LCD \& optimal & $[12,3,4]$ \\
\hline
$S_{10}$ & $[12,2,7]$ & non-LCD & $[12,10,1]$ \\
\hline
\end{longtable}}

{\centering
\begin{longtable}{|c|}
\hline 
\textbf{Basis $S_{i}$} \\ 
\hline \hline
$S_{1} = \{D(\zeta^{2}), D(\zeta^{4}), D(\zeta^{6}), D(\zeta^{8}), D(\zeta^{10}), D(\zeta^{11})\}$ \\ 
\hline 
$S_{2} = \{D(\zeta^{2}), D(\zeta^{4}), D(\zeta^{5}), D(\zeta^{6}), D(\zeta^{8}), D(\zeta^{10}), D(\zeta^{11})\}$ \\ 
\hline 
$S_{3} = \{D(\zeta^{2}), D(\zeta^{4}), D(\zeta^{6}), D(\zeta^{8}), D(\zeta^{10})\}$ \\ 
\hline 
$S_{4} = \{D(\zeta^{2}), D(\zeta^{4}), D(\zeta^{6})\}$ \\ 
\hline 
$S_{5} = \{D(\zeta^{2}), D(\zeta^{4}), D(\zeta^{6}), D(\zeta^{8})\}$ \\ 
\hline
$S_{6} = \{D(\zeta), D(\zeta^{5}), D(\zeta^{6}), D(\zeta^{11})\}$ \\ 
\hline
$S_{7} = \{D(\zeta), D(\zeta^{2}), D(\zeta^{5}), D(\zeta^{6}), D(\zeta^{9}), D(\zeta^{11})\}$ \\ 
\hline
$S_{8} = \{D(\zeta), D(\zeta^{2}), D(\zeta^{3}), D(\zeta^{5}), D(\zeta^{6}), D(\zeta^{8}), D(\zeta^{9}), D(\zeta^{11})\}$ \\ 
\hline 
$S_{9} = \{D(\zeta), D(\zeta^{2}), D(\zeta^{3}), D(\zeta^{5}), D(\zeta^{6}), D(\zeta^{8}), D(\zeta^{9}), D(\zeta^{10}), D(\zeta^{11})\}$ \\ 
\hline 
$S_{10} = \{D(\zeta), D(\zeta^{10})\}$ \\ 
\hline 
\end{longtable}}

A $5 \times 12$ generator matrix for the binary $[12,5,4]$ non-LCD optimal Hom-IDD code obtained using the basis $S_{3}$ is $$\begin{pmatrix}
1 & 1 & 1 & 0 & 1 & 1 & 1 & 1 & 1 & 1 & 1 & 0 \\
1 & 0 & 0 & 0 & 1 & 1 & 0 & 1 & 0 & 0 & 0 & 1 \\
0 & 0 & 1 & 1 & 0 & 1 & 1 & 0 & 0 & 1 & 1 & 0 \\
1 & 1 & 0 & 1 & 0 & 1 & 1 & 0 & 0 & 0 & 1 & 1 \\
1 & 0 & 0 & 1 & 1 & 1 & 1 & 0 & 1 & 1 & 1 & 0
\end{pmatrix}.$$
\end{example} \vspace{10pt}

\begin{example}
Let $\mathbb{K} = \mathbb{Q}(\zeta)$, $\zeta$ primitive $n^{\text{th}}$-root of unity, where $n = p^{5}$ and $p=2$. Let $\sigma(\zeta) = \zeta$ and $\tau(\zeta) = \zeta^{3}$. Then by Theorem \ref{theorem 3.2} (in particular, Corollary \ref{corollary 3.7}), the map $D:O_{\mathbb{K}} \rightarrow O_{\mathbb{K}}$ given by $$D(\zeta) = 1 + \zeta + \zeta^{2} + \zeta^{5} + \zeta^{7}$$ is a $(\sigma, \tau)$-derivation of $O_{\mathbb{K}} = \mathbb{Z}[\zeta]$. Also, $[\mathbb{K} : \mathbb{Q}] = \phi(n) = \phi(32) = 16$ and the $n^{\text{th}}$ cyclotomic polynomial is $\Phi_{n}(x) = \Phi_{32}(x) = x^{16} + 1$.

Below, we have obtained binary Hom-IDD codes of length $\phi(n) = \phi(32) = 16$ over the finite field $\mathbb{Z}_{2}$ with various parameters.

{\centering
\begin{longtable}{|c|c|c|c|}
\hline 
\textbf{Basis $S_{i}$} & \thead{\textbf{Code Description:}\\ $[n',k',d']$} & \textbf{Code Properties} & \thead{\textbf{Dual Code} \\ \textbf{Description:} \\ $[n',k',d']$} \\ 
\hline \hline
$S_{1}$ & $[16,11,1]$ & non-LCD & $[16,5,4]$ \\
\hline
$S_{2}$ & $[16,5,6]$ & non-LCD & $[16,11,2]$ \\
\hline
$S_{3}$ & $[16,6,4]$ & non-LCD & $[16,10,3]$ \\
\hline
$S_{4}$ & $[16,6,5]$ & LCD & $[16,10,1]$ \\
\hline
$S_{5}$ & $[16,7,5]$ & non-LCD & $[16,9,3]$ \\
\hline
$S_{6}$ & $[16,8,4]$ & non-LCD & $[16,8,4]$ \\
\hline
$S_{7}$ & $[16,9,3]$ & non-LCD & $[16,7,4]$ \\
\hline
$S_{8}$ & $[16,8,3]$ & non-LCD & $[16,8,4]$ \\
\hline
$S_{9}$ & $[16,5,5]$ & LCD & $[16,11,1]$ \\
\hline
$S_{10}$ & $[16,6,5]$ & non-LCD & $[16,10,2]$ \\
\hline
$S_{11}$ & $[16,4,6]$ & non-LCD & $[16,12,1]$ \\
\hline
$S_{12}$ & $[16,3,6]$ & non-LCD & $[16,13,1]$ \\
\hline
\end{longtable}}

{\centering
\begin{longtable}{|c|}
\hline 
\textbf{Basis $S_{i}$} \\ 
\hline \hline
$S_{1} = \{D(\zeta), D(\zeta^{2}), D(\zeta^{3}), D(\zeta^{4}), D(\zeta^{5}), D(\zeta^{6}), D(\zeta^{7}), D(\zeta^{9}), D(\zeta^{10}), D(\zeta^{11}), D(\zeta^{13})\}$ \\ 
\hline 
$S_{2} = \{D(\zeta^{2}), D(\zeta^{4}), D(\zeta^{6}), D(\zeta^{10}), D(\zeta^{11})\}$ \\ 
\hline 
$S_{3} = \{D(\zeta^{2}), D(\zeta^{4}), D(\zeta^{6}), D(\zeta^{10}), D(\zeta^{11}), D(\zeta^{13})\}$ \\ 
\hline 
$S_{4} = \{D(\zeta), D(\zeta^{3}), D(\zeta^{5}), D(\zeta^{7}), D(\zeta^{13}), D(\zeta^{15})\}$ \\ 
\hline
$S_{5} = \{D(\zeta), D(\zeta^{3}), D(\zeta^{5}), D(\zeta^{7}), D(\zeta^{9}), D(\zeta^{13}), D(\zeta^{15})\}$ \\ 
\hline
$S_{6} = \{D(\zeta), D(\zeta^{3}), D(\zeta^{5}), D(\zeta^{7}), D(\zeta^{8}), D(\zeta^{9}), D(\zeta^{13}), D(\zeta^{15})\}$ \\ 
\hline 
$S_{7} = \{D(\zeta), D(\zeta^{3}), D(\zeta^{5}), D(\zeta^{7}), D(\zeta^{8}), D(\zeta^{9}), D(\zeta^{12}), D(\zeta^{13}), D(\zeta^{15})\}$ \\ 
\hline 
$S_{8} = \{D(\zeta), D(\zeta^{3}), D(\zeta^{4}), D(\zeta^{5}), D(\zeta^{7}), D(\zeta^{8}), D(\zeta^{9}), D(\zeta^{11})\}$ \\ 
\hline 
$S_{9} = \{D(\zeta^{7}), D(\zeta^{9}), D(\zeta^{11}), D(\zeta^{13}), D(\zeta^{15})\}$ \\ 
\hline 
$S_{10} = \{D(\zeta^{2}), D(\zeta^{4}), D(\zeta^{6}), D(\zeta^{9}), D(\zeta^{10}), D(\zeta^{11})\}$ \\ 
\hline 
$S_{11} = \{D(\zeta^{2}), D(\zeta^{4}), D(\zeta^{6}), D(\zeta^{10})\}$ \\ 
\hline
$S_{12} = \{D(\zeta^{4}), D(\zeta^{6}), D(\zeta^{10})\}$ \\ 
\hline 
\end{longtable}}

A $6 \times 16$ generator matrix for the binary $[16,6,5]$ LCD Hom-IDD code obtained using the basis $S_{4}$ is $$\begin{pmatrix}
1 & 1 & 1 & 0 & 0 & 1 & 0 & 1 & 0 & 0 & 0 & 0 & 0 & 0 & 0 & 0 \\
0 & 0 & 1 & 1 & 0 & 1 & 0 & 0 & 1 & 0 & 0 & 0 & 0 & 1 & 0 & 0 \\
0 & 0 & 0 & 1 & 1 & 1 & 0 & 1 & 0 & 0 & 0 & 1 & 0 & 1 & 1 & 0 \\
0 & 1 & 0 & 1 & 1 & 0 & 1 & 1 & 0 & 0 & 0 & 0 & 0 & 1 & 0 & 1 \\
0 & 1 & 0 & 0 & 0 & 0 & 1 & 1 & 0 & 1 & 0 & 0 & 1 & 0 & 0 & 0 \\
0 & 1 & 0 & 1 & 0 & 0 & 0 & 0 & 0 & 0 & 0 & 0 & 1 & 1 & 1 & 0
\end{pmatrix}.$$
\end{example}

\section{Conclusion}\label{section 7}
This article studied the inner and outer twisted derivations in algebraic number fields. We have studied the twisted derivations of some special integral extensions of a commutative unital ring. Given a commutative unital ring $A$, its integral extension $B = A[\theta]$, which is also an integral domain, the minimal splitting field $\mathbb{E}$ of $\theta \in B$ over the quotient field $\mathbb{K}$ of $B$, and a pair of distinct $A$-algebra homomorphisms $\sigma, \tau: B \rightarrow \mathbb{E}$, we have classified all $A$-linear maps $D: B \rightarrow \mathbb{E}$ which are $(\sigma, \tau)$-derivations. Consequently, we have obtained important results regarding twisted derivations in algebraic number fields and their ring of algebraic integers. Then we have conjectured characterizations for a $(\sigma, \tau)$-derivation $D:O_{\mathbb{K}} \rightarrow O_{\mathbb{K}}$ to be inner, where $O_{\mathbb{K}}$ is the ring of algebraic integers of an $n^{\text{th}}$-cyclotomic number field $\mathbb{K} = \mathbb{Q}(\zeta)$ ($\zeta$ a primitive $n^{\text{th}}$-root of unity) when $n$ is of the forms $n=2^{r}p$ ($r \in \mathbb{N}$, $p$ an odd rational prime) and $n = p^{k}$ ($k \in \mathbb{N} \setminus \{1\}$, $p$ any rational prime).
Then, as an application of our main Theorem \ref{theorem 3.2} (in particular, Corollary \ref{corollary 3.7}) and also conjectures \ref{conjecture 4.2} and \ref{conjecture 4.4}, we have conjectured the existence and non-existence of non-zero outer derivations $D: O_{\mathbb{K}} \rightarrow O_{\mathbb{K}}$ for the above two forms of $n$, thus answering the twisted derivation problem in $O_{\mathbb{K}}$.  Finally, we have given another application of our main Theorem \ref{theorem 3.2} in the construction of codes in coding theory. This research, therefore, opens a way for future research into twisted derivations in algebraic number fields, which can prove to be highly useful from the viewpoint of their vast applications in mathematics and other sciences, especially coding theory.\vspace{10pt}

\noindent \textbf{Acknowledgements}\vspace{8pt}

\noindent The second author was the ConsenSys Blockchain chair professor when the work was being carried out. He thanks ConsenSys AG for that privilege.\vspace{10pt}

\noindent \textbf{Declaration of Interest Statement}\vspace{8pt}

\noindent The authors report that there are no competing interests to declare.

\bibliographystyle{plain}
\bibliography{references}

\end{document}